\documentclass[reqno]{amsart}
\usepackage{amsthm,amsmath}
\usepackage{amscd, amssymb, amsfonts}
\usepackage{hyperref}

\usepackage{verbatim}
\usepackage{bold-extra}
\usepackage[all]{xy}
\usepackage{xspace}
\usepackage{mathrsfs}
\usepackage{enumerate}
\usepackage{todonotes}
\usepackage{xcolor}

\newcommand{\R}{\ensuremath{\mathbb{R}}\xspace}
\newcommand{\C}{\ensuremath{\mathbb{C}}\xspace}

\newcommand{\nt}{\ensuremath{\mathbb{N}}\xspace}
\newcommand{\inc}{\ensuremath{\hookrightarrow}}

\newcommand{\pr}[2]{\ensuremath{\langle {#1},{#2}\rangle}}
\newcommand{\norm}[1]{\ensuremath{\|#1\|}}

\newcommand{\mn}{\ensuremath{\mathcal{N}}\xspace}
\newcommand{\msn}{\ensuremath{\mathcal{SN}}\xspace}

\newcommand{\bts}[3]{\ensuremath{{#1}\bigotimes_{#2}{#3}}\xspace}
\newcommand{\bats}[2]{\ensuremath{{#1}\bigodot{#2}}\xspace}
\newcommand{\tmin}[2]{\ensuremath{{#1}\bigotimes_{\min}{#2}}\xspace}
\newcommand{\tmax}[2]{\ensuremath{{#1}\bigotimes_{\max}{#2}}\xspace}

\newcommand{\link}[1]{\ensuremath{\mathbb{L}({#1})}}

\newcommand{\al}{\ensuremath{\alpha}\xspace} 
 
\newcommand{\ga}{\ensuremath{\gamma}\xspace}
\newcommand{\la}{\ensuremath{\lambda}\xspace}

\newcommand{\te}{\ensuremath{\theta}\xspace}

\newcommand{\Te}{\ensuremath{\Theta}\xspace}

\newcommand{\ov}[1]{\overline{#1}}

\newcommand{\ma}{\ensuremath{\mathcal{A}}\xspace}
\newcommand{\mb}{\ensuremath{\mathcal{B}}\xspace}
\newcommand{\me}{\ensuremath{\mathcal{E}}\xspace}

\newcommand{\cs}{$C^*$-algebra\xspace}
\newcommand{\css}{$C^*$-algebras\xspace}
\newcommand{\scs}{$C^*$-subalgebra\xspace}

\newcommand{\st}{$*$-tring\xspace}

\newcommand{\ct}{$C^*$-tring\xspace}
\newcommand{\cts}{$C^*$-trings\xspace}
\newcommand{\hm}{homomorphism\xspace}
\newcommand{\hms}{homomorphisms\xspace}

\newcommand{\es}{exact sequence\xspace}

\newcommand{\bc}{\begin{center}}
\newcommand{\ec}{\end{center}}
\newcommand{\be}{\begin{enumerate}}
\newcommand{\ee}{\end{enumerate}}
\newcommand{\bi}{\begin{itemize}}
\newcommand{\ei}{\end{itemize}}
\newcommand{\bd}{\begin{description}}
\newcommand{\ed}{\end{description}}
\newcommand{\beq}{\begin{equation}}
\newcommand{\eeq}{\end{equation}}
\newcommand{\beqa}{\begin{eqnarray}}
\newcommand{\eeqa}{\end{eqnarray}}
\newcommand{\bfr}{\begin{flushright}}
\newcommand{\efr}{\end{flushright}}
\newcommand{\bfl}{\begin{flushleft}}
\newcommand{\efl}{\end{flushleft}}
\newcommand{\bp}{\begin{picture}}
\newcommand{\ep}{\end{picture}}

\DeclareMathOperator{\gen}{span}

\DeclareMathOperator{\emo}{End}
\DeclareMathOperator{\Hom}{Hom}

\newtheorem{theorem}{Theorem}[section]
\newtheorem{lemma}[theorem]{Lemma}
\newtheorem{proposition}[theorem]{Proposition}
\newtheorem{corollary}[theorem]{Corollary}
\theoremstyle{definition}
\newtheorem{definition}[theorem]{Definition}
\newtheorem{example}[theorem]{Example}

\theoremstyle{remark}
\newtheorem{remark}[theorem]{Remark}
\numberwithin{equation}{section}

\title[Applications of ternary rings to $C^*$-algebras]{Applications of ternary rings to $C^*$-algebras}
\author{Fernando Abadie}
\author{Dami\'an Ferraro}
\date{\today}

\address{Centro de Matem\'atica, Facultad de Ciencias,
              Universidad de la Rep\'ublica, Igu\'a 4225, CP 11400,
              Montevideo, Uruguay}
\email{fabadie@cmat.edu.uy}

\address{Departamento de Matem\'atica y Estad\'istica del Litoral,
           Universidad de la Rep\'ublica, Rivera 1350, CP 50000,
           Salto, Uruguay}
\email{dferraro@unorte.edu.uy}


\makeindex
\begin{document}

\begin{abstract}
We show that there is a functor from the category of positive admissible ternary rings to the category of $*$-algebras, which induces an isomorphism of partially ordered sets between the families of $C^*$-norms on the ternary ring and its corresponding $*$-algebra.
We apply this functor to obtain Morita-Rieffel equivalence results between cross sectional $C^*$-algebras of Fell bundles, and to extend the theory of tensor products of \css to the larger category of full Hilbert $C^*$-modules. We prove that, like in the case of \css, there exist maximal and minimal tensor products. 
As applications we give simple proofs of the invariance of nuclearity and exactness under Morita-Rieffel equivalence of \css.
\end{abstract}

\maketitle

\tableofcontents

\section{Introduction}\label{sec:intro}

\par An important tool in the study of \css is Morita-Rieffel
equivalence.
When two \css are Morita-Rieffel equivalent, they are
related by a certain type of bimodule, from which one can see that 
these algebras share many properties. A Morita-Rieffel equivalence
between two \css implies that these algebras 
have many characteristics in common: they have the same  
K-theory, their spectra and primitive ideal spaces are homeomorphic,
etc. In \cite{z2}, Zettl introduced and studied $C^*$-ternary rings,
and showed that these objects are essentially Morita-Rieffel
equivalence bimodules. In fact, given a $C^*$-ternary ring $E$, there
exists essentially a unique structure of Morita-Rieffel equivalence
bimodule on $E$ compatible with its structure of ternary ring (perhaps
after a minor change on the ternary product).      
\par On the other hand, when dealing with constructions such as tensor
products or any sort of crossed products of \css, in general one 
has to follow two steps: first one defines some $*$-algebra, and then
one takes the completion of that algebra with respect to a $C^*$-norm. A
situation that appears frequently is that there is more than one
reasonable $C^*$-norm to perform this second step. In many cases, for
instance in several imprimitivity theorems, one is interested in
finding a Morita-Rieffel equivalence between different
$C^*$-completions of a given pair of $*$-algebras which are related by
a certain bimodule. This is the situation we study in the
present paper, adopting a viewpoint similar to that in Zettl's work,
but starting from a more algebraic level. 
\par More precisely, suppose $E$ is an $A-B$ bimodule,
where $A$ and $B$ are $*$-algebras, $\pr{\,}{\,}_A:E\times E\to A$ and
$\pr{\,}{\,}_B:E\times E\to B$ satisfy all the algebraic properties
of Hilbert bimodule inner products. In particular
$\pr{x}{y}_Az=x\pr{y}{z}_B$, $\forall x,y,z\in E$. Then we can endow
$E$ with a $*$-ternary ring structure by defining 
a ternary product $(\,,\,,\,):E\times E\times E \to E$ such that 
$(x,y,z)=x\pr{y}{z}_B$. We show that, under certain conditions, the
partially ordered sets of $C^*$-norms on $E$ and on the $*$-algebras
$A$ and $B$ are 
isomorphic to each other, in such a way that the completions with
respect to corresponding $C^*$-norms under these isomorphisms yields a
Morita-Rieffel equivalence bimodule.      

We think that the best way to do it is by using the above mentioned  
abstract characterization of equivalence bimodules given by Zettl in \cite{z2},
under the name of $C^*$-ternary rings. Such an object is a Banach space
with a ternary product on it, which implicitly carries all the
structure of an equivalence bimodule. Natural morphisms between
$C^*$-ternary rings are linear maps that preserve ternary products. With
such morphisms, one obtains a $C^*$-category, which is very convenient
for the study of properties invariant under Morita-Rieffel equivalence.     
\par The structure of the paper is as follows.  In the next section,
working in a pure algebraic level, we define the category of
admissible $*$-ternary rings, and we show there is a functor from
this category to the category of $*$-algebras or, more precisely, to
the category of right basic triples (see Definition~\ref{defn:bt}). 
In Section~3, given an admissible ternary ring $E$ with associated
basic triple $(E,A,\pr{\,}{\,}_A)$, we consider the 
lattice of $C^*$-seminorms on $A$ that satisfy the Cauchy-Schwarz
inequality
$\norm{\pr{x}{y}_A}^2\leq\norm{\pr{x}{x}}\,\norm{\pr{y}{y}}$, $\forall
x,y\in E$. Then we prove that this lattice is isomorphic to the 
lattice of $C^*$-seminorms on $E$. 
In passing we obtain some of the results of \cite{z2} and
\cite{env} regarding $C^*$-ternary rings. Besides, since there is also
a functor to the category of left basic triples, we obtain a fortiori
an isomorphism between the lattices of $C^*$-seminorms (satisfying the 
Cauchy-Schwarz property) on the $*$-algebras associated to the left and to the right sides. The
Hausdorff completions 
of corresponding $C^*$-seminorms under this isomorphism turn out to be 
Morita-Rieffel equivalent. In the last part of Section~3 we consider
positive ternary rings, for which the $C^*$-seminorms on the associated
$*$-algebras automatically satisfy the Cauchy-Schwarz
inequality. In Section~4 we briefly study the case of
$C^*$-ternary rings, in which basic triples are replaced by
$C^*$-basic triples, that is, Hilbert modules, and the functors
from $C^*$-ternary rings to $C^*$-basic triples are shown to be
exact. Finally, Section~5 is devoted to applications. We first refine
a result from \cite{flau} concerning cross sectional algebras of Fell
bundles over groups. Then we consider tensor products of $C^*$-ternary
rings, which is essentially the same as tensor products of Hilbert
modules. We show that the theory of tensor products of $C^*$-algebras
extends to this larger category, in the sense that there exist a
maximal and a minimal tensor products. By using this theory we obtain
easy and natural proofs of the known results of the Morita-Rieffel
invariance of nuclearity and exactness of $C^*$-algebras.

\section{Ternary rings}
\subsection{Ternary rings}
\begin{definition}\label{defn:1} 
A $*$-ternary ring is a complex linear space $E$ with a map $\mu
:E\times E\times E\to E$, called $*$-ternary product on $E$, which is 
linear in the odd variables and conjugate linear in the second one,
and such that:   
\begin{equation*}
    \mu\big(\mu(x,y,z),u,v\big)
      =\mu\big(x,\mu(u,z,y),v\big)
      =\mu\big(x,y,\mu(z,u,v)\big),
    \ \forall x,y,z,u,v\in E
\end{equation*} 
A \hm of $*$-ternary rings is a linear map $\phi :(E,\mu )\to (F,\nu )$
such that $ \nu \big(\phi (x),\phi (y),\phi (z)\big)
        =\phi\big(\mu(x,y,z)\big),\, \forall
                                    x,y,z\in E.$ 
Sometimes we will write $(x,y,z)$ or $(x,y,z)_E$ instead of 
$\mu (x,y,z)$, and we will use the expression $*$-tring instead of
$*$-ternary ring. 
\end{definition}

\par There is an inclusion of the category of $*$-algebras into the
category of $*$-trings: if $A$ is a $*$-algebra, then $(x,y,z)\mapsto
xy^*z$ is a ternary product on $A$, and if $\pi:A\to A'$ is a
homomorphism of $*$-algebras, then so is of $*$-trings.  

\begin{definition}\label{defn:hermetic}
If a subspace $F$ of a $*$-tring $E$ is invariant under the ternary
product, we say that it is a sub-$*$-tring of $E$, or just a subring
of $E$. A subring $F$ is said to be hermetic in $E$ if for $x\in
E$ we have $(x,x,x)\in F\iff x\in F$.    
\end{definition} 

\begin{definition}\label{defn:admiss}
A $*$-tring $E$ will be called admissible if $\{0\}$ is hermetic in
$E$.  A $*$-algebra $A$ will be called admissible if it is admissible
as a $*$-tring. 
\end{definition} 
\par Note that a $*$-algebra $A$ is admissible if
and only if the condition $a^*a=0$ implies $a=0$. 

\begin{definition}\label{defn:*ideals}
Let $E$ be a $*$-tring and $F\subseteq E$ a subspace. We say that 
$F$ is an ideal of $E$ if  $(E,E,F)+(E,F,E)+(F,E,E)\subseteq F$. 
\end{definition}
\par If $\pi:E\to F$ is a homomorphism into an admissible $*$-tring
$F$, then $\ker\pi$ is an hermetic ideal of $E$: 
\begin{equation*}
  \pi((x,x,x))=0\iff (\pi(x),\pi(x),\pi(x))=0\iff \pi(x)=0
\end{equation*}
\par In case $F$ is an ideal of $E$, then $E/F$ has an obvious
structure of $*$-tring for which the canonical map $q:E\to E/F$ is a
homomorphism of $*$-trings. Note that $E/F$ is admissible whenever $F$
is hermetic. In particular if $\pi:E\to F$ is a homomorphism into an
admissible $*$-tring $F$, then $E/\ker\pi$ is admissible

\par Suppose $E$ is a complex vector space, and let $E^*$ denote
its complex conjugate linear space. If $(E,\mu )$ is a $*$-tring, then
$\mu^*:E^*\times E^*\times E^*\to E^*$ given by $\mu^*(x,y,z)=\mu
(z,y,x)$, $\forall x,y,z\in E^*$, is a $*$-ternary product on $E^*$.  
We call $(E^*,\mu^*)$ the \textit{adjoint} or \textit{reverse $*$-tring}
of $(E,\mu )$. If $\pi:E\to F$ is a homomorphism, then $\pi$
remains a homomorphism $E^*\to F^*$, so it is clear that reversion is
an autofunctor of order two of the category of $*$-trings, which
moreover sends admissible $*$-trings into admissible $*$-trings. If $A$ is
a $*$-algebra considered as a $*$-tring as above, then its reverse $*$-tring
$A^*$ is the conjugate linear space of $A^{\text{op}}$ considered as a
$*$-tring.
\begin{example}[{\bf Basic triples}]\label{exmp:0} Suppose
  $(E,A,\pr{\,}{\,})$ is a triple  
consisting of a $\C$-vector space $E$, a $*$-algebra $A$ over which
$E$ is a right module, and 
a sesquilinear map $\pr{\,}{\,}:E\times E\to A$ (conjugate linear in
the first variable), such that $\pr{x}{y}a=\pr{x}{ya}$ and
$\pr{x}{y}^*=\pr{y}{x}$, $\forall x,y\in E$, $a\in A$. Then
$(\,,\,,\,):E\times E\times E\to E$ given by $(x,y,z)\mapsto
x\pr{y}{z}$ is a ternary product. We will say that $(\,E,
(\,,\,,\,)\,)$ is the ternary ring associated with
$(E,A,\pr{\,}{\,})$. 
\end{example}
\begin{definition}\label{defn:bt}
Triples as in Example~\ref{exmp:0} will be referred to as (right) basic
triples. A basic triple $(E,A,\pr{\,}{\,}_A)$ will be called
admissible whenever $A$ is admissible, and full if
$\textrm{span}\{\pr{x}{y}_A:\, x,y\in 
E\}=A$. By a homomorphism from the basic triple $(E,A,\pr{\,}{\,}_A)$
into the basic triple $(F,B,\pr{\,}{\,}_B)$ we mean a pair
$(\varphi,\psi)$ of maps such that $\varphi:E\to F$ is linear,
$\psi:A\to B$ is a homomorphism of $*$-algebras, and
$\varphi(xa)=\varphi(x)\psi(a)$, $\forall x\in E$, $a\in A$.
\par Similarly we can define left basic triples, using left
instead of right $A$-modules.  
\end{definition}

We will see soon that any admissible $*$-tring can be described in terms
of basic triples as in \ref{exmp:0}. 

\begin{proposition}\label{prop:adfaith} Let $(E,A,\pr{\,}{\,})$ be a basic
  triple. 
\be
  \item If $A$ is admissible, and $\pr{x}{x}=0$ implies $x=0$, then
    the $*$-tring $E$ is admissible as well. 
  \item If $(E,A,\pr{\,}{\,})$ is admissible and full, then $E$ is 
    faithful as an $A$-module.   
\ee
\end{proposition}
\begin{proof}
If $x\in E$ is such that $x\pr{x}{x}=0$, then 
\begin{equation*}
  \pr{x}{x}^*\pr{x}{x}
    =\pr{x}{x}\pr{x}{x}
    =\pr{x}{x\pr{x}{x}}
    =0.
\end{equation*}
Now, if 
$A$ is admissible, the latter equality implies
$\pr{x}{x}=0$, so $x=0$. As for the 
second statement suppose $(E,A,\pr{\,}{\,}_A)$ is 
  admissible and full, and
  $a\in A$ is such that $a=\sum_{j=1}^n\pr{y_j}{z_j}$ and $ya=0$,
  $\forall y\in F$. Then we have $a^*a=\sum_{j=1}^n\pr{y_j}{z_ja}=0$, so
  $a=0$. Then $E$ is a faithful $A$-module.    
\end{proof}

\begin{lemma}\label{lem:0}
Suppose that $(E,A,\pr{\,}{\,}_A)$ and  
$(F,B,\pr{\,}{\,}_B)$ are basic triples, with the former full, and $F$ admissible as $*$-tring and faithful as a $B$-module.
 Then, if   
$\varphi:(\,E,(\,,\,,\,)\,)\to (\,F,(\,,\,,\,)\,)$ is a homomorphism
between their associated $*$-trings, there exists a unique
homomorphism of $*$-algebras   
$\psi:A\to B$ such that $\psi(\pr{x}{y}_A)=\pr{\varphi(x)}{\varphi(y)}_B$,
$\forall x,y\in E$. Besides we have $\varphi(xa)=\varphi(x)\psi(a)$,
$\forall x\in E$, $a\in A$, and 
\begin{equation}\label{eq:ker} \ker\psi\subseteq\{a\in A:
Ea\subseteq\ker\varphi\}\subseteq\{a\in
A:\psi(a)^*\psi(a)=0\},\end{equation} both  
inclusions being equalities if $B$ is admissible. If $E$ is also a
faithful $A$-module and $\varphi$ is injective, then so is $\psi$.       
\end{lemma} 
\begin{proof}
\par We will suppose that $(F,B,\pr{\,}{\,}_B)$ is full: otherwise
we just replace $B$ by $\textrm{span}\pr{F}{F}_B$. We concentrate in
showing the existence of the 
map $\psi$, because its uniqueness is obvious. To this end suppose
that $x_1,\ldots,x_n$ and $y_1,\ldots,y_n$ 
are elements in $E$ such that $\sum_{j=1}^n\pr{x_j}{y_j}_A=0$, and
therefore also $\sum_{j=1}^n\pr{y_j}{x_j}_A=0$. Consider the element
$c:=\sum_{j=1}^n\pr{\varphi(x_j)}{\varphi(y_j)}_B$ of $B$. All we have to do
is to show that $c=0$. Now, if $x\in E$ and $u\in
F$ we have 
\begin{gather*} 
(\varphi (x),uc,uc) =\sum_k(\varphi(x), (u,\varphi(x_k),\varphi(y_k)) ,uc) \\
=\sum_k((\varphi(x),\varphi(y_k),\varphi(x_k)) ,u,uc) 
=(\varphi(x\sum_k\pr{y_k}{x_k}_A),u,uc) =0.
\end{gather*}
Hence, if $u\in F$: 
\begin{gather*}
(uc,uc,uc) =\sum_j((u,\varphi (x_j),\varphi (y_j)) ,uc,uc) 
=\sum_j(u,\varphi (x_j),(\varphi (y_j),uc,uc) ) =0
\end{gather*}
Since $F$ is admissible, it follows that $uc=0, \forall u\in F$, so
$c=0$ because $F$ is a faithful $B$-module.   
\par Suppose now that $a\in\ker\psi$. Then
$\varphi(xa)=\varphi(x)\psi(a)=0$, so
$Ea\subseteq\ker\varphi$. On the other hand, if the element
$a=\sum_j\pr{x_j}{y_j}_A$ 
is such that $Ea\subseteq \ker\varphi$, then
$\psi(a^*a)=\psi(\sum_{i,j}\pr{y_i}{x_i\pr{x_j}{y_j}_A}_A)=\sum_{i,j}\pr{\varphi(y_i)}{\varphi(x_i)\psi(\pr{x_j}{y_j}_A)}_A$,
so 
\begin{equation*}
  \psi(a)^*\psi(a)
  =\sum_i\pr{\varphi(y_i)}{\varphi(x_i)\psi(\sum_j\pr{x_j}{y_j}_A)}_A
  =\sum_i\pr{\varphi(y_i)}{\varphi(x_ia)}_A
  =0,
\end{equation*}
because $\varphi(x_ia)=0$ $\forall i$. In case $B$ is admissible we
have $\psi(a)^*\psi(a)=0$ if and only if 
$a\in\ker\psi$, so in this case the three considered sets
agree. Finally, when $E$ is faithful and $\ker\varphi=0$, we have 
$\{a\in A:Ea\subseteq\ker\varphi\}=0$, so $\ker\psi=0$. 
 \end{proof}

\par Given two modules $E$ and $F$ over a ring $R$, we denote by
$\Hom_R(E,F)$ the abelian group of $R$-linear maps from $E$ into $F$,
and just by $\emo_R(E)$ in case $E=F$. Let $E$ be   
an admissible $*$-tring, and suppose  
$T\in\emo_\C(E)$ is such that there exists $S\in\emo_\C(E)$ that
satisfies $(x,Ty,z)=(Sx,y,z), \forall x,y,z\in E$. Since $\{0\}$ is
hermetic in $E$, given $T\in \emo_\C(E)$, there exists at most one
such endomorphism $S$; in this case we say that $S$ is the 
adjoint of $T$ to the left, and we denote it by $T^*$. The set
$\mathscr{L}_l(E)$ of $\C$-linear endomorphisms of $E$ that have an
adjoint to the left is clearly a unital 
subalgebra of $\emo_\C(E)$. Every pair of elements $y,z\in E$ gives
rise to an endomorphism $\theta_{y,z}:E\to E$ given by
$\theta_{y,z}(x):=(x,y,z)$. It is readily checked that $\theta_{y,z}$
is adjointable with adjoint~$\theta_{z,y}$.  

\begin{proposition}\label{prop:adjointable}
Let $E$ be an admissible $*$-tring. Then the map
$*:\mathscr{L}_l(E)\to \mathscr{L}_l(E)$, given by taking the adjoint, 
is an involution in $\mathscr{L}_l(E)$. Moreover, the $*$-algebra
$\mathscr{L}_l(E)$ is an admissible $*$-tring, and
$\textrm{span}\{\theta_{y,z}:y,z\in E\}$ is a twosided
ideal of $\mathscr{L}_l(E)$, which is essential in the sense that
$T\theta_{y,z}=0$ $\forall y,z\in E$ or $\theta_{y,z}T=0$ $\forall
y,z\in E$ implies $T=0$.       
\end{proposition}
\begin{proof}
It is clear that the map $T\mapsto T^*$ is conjugate linear and
antimultiplicative. On the other hand, if $T\in \mathscr{L}_l(E)$:   
\begin{gather*}
(u,(Tx,y,z),u)
=(u,z,(y,T(x),u))
=(u,z,(T^*(y),x,u))
=(u,(x,T^*(y),z),u)
\end{gather*}
$\forall x,y,z,u\in E$ and $T\in\mathscr{L}_l(E)$, which shows that
$T^{**}=T$. Now, if $x\in E$, and $T\in \mathscr{L}_l(E)$ is such
that $T^*T=0$: $(Tx,Tx,Tx)=(x,T^*Tx,Tx)=0$, so $T(x)=0$, and therefore 
$T=0$. Finally, if $T\in\mathscr{L}_l(E)$ and
$x,y,z\in E$: 
$\theta_{y,z}T(x)=(Tx,y,z)=(x,T^*y,z)=\theta_{T^*y,z}(x)$. Thus 
$T\theta_{y,z}=(\theta_{z,y}T^*)^*=\theta_{y,Tz}$. This shows that 
$\gen\{\theta_{y,z}:y,z\in E\}$ is an ideal of
$\mathscr{L}_l(E)$. If $\theta_{y,z}T=0$ $\forall y,z\in E$, then
$0=\theta_{Tx,Tx}T(x)=(Tx,Tx,Tx)$, $\forall x\in E$. Then $Tx=0$
$\forall x\in E$ because $E$ is admissible, so $T=0$.      
 \end{proof}
\par The next result shows that any admissible $*$-tring $E$ gives
rise to an admissible and full {\it right} basic triple
$(E,E_0^r,\pr{\,}{\,}_r)$. In the same way one could show that $E$
also defines a {\it left} basic triple $(E,E_0^l,\pr{\,}{\,}_l)$.      

\begin{theorem}\label{thm:*level}
Let $E$ and $F$ be admissible $*$-trings. Then:  
\be
   \item There exists a pair $(E_0^r,\pr{\ }{\ }_r)$ such that
     $(E,E_0^r,\pr{\ }{\ }_r)$ is an admissible and full basic
     triple, whose associated $*$-tring is $E$.
 \item If $\pi:E\to F$ is a \hm of $*$-trings, and
   $(E^r_0,\pr{\cdot}{\cdot}_r)$ and $(F^r_0,\pr{\cdot}{\cdot}_r)$ are
   pairs like above for $E$ and $F$ respectively, there exists a
   unique \hm of $*$-algebras $\pi^r_0:E^r_0\to F^r_0$ such that
   \begin{equation*}
    \pi^r_0(\pr{x}{y}_r)=\pr{\pi (x)}{\pi (y)}_r, \ \forall x,y\in E.
   \end{equation*} Moreover, $\pi (xb)= \pi (x)\pi^r_0(b)$, $\forall x\in E$,
   $b\in E^r_0$, that is, the pair $(\pi,\pi_0^r)$ is a homomorphism
   of basic triples.  
 \item The pair $(E^r_0,\pr{\cdot}{\cdot}_r)$ is the unique (up to
   canonical isomorphisms) such that the triple $(E,E^r_0,\pr{\cdot}{\cdot}_r)$
   is a full and admissible with $E$ as associated
   $*$-tring.     
\ee
\end{theorem}
\begin{proof}
Note that $E$ is a faithful right
$\mathscr{L}_l(E)^{\textrm{op}}$-module with $xT:=T(x)$. Consider the
ideal $E^r_0:=\gen\{\te_{y,z}:\, y,z\in  
E\}$ of $\mathscr{L}_l(E)^{\textrm{op}}$ and let
$\pr{\ }{\ }_r:E\times E\to E_0^r$ be given by
$\pr{x}{y}_r:=\theta_{x,y}$. It is routine to verify that 
$(\,E,E_0^r,\pr{\ }{\ }_r\,)$ is a full and admissible basic triple whose
associated $*$-tring is $E$. 
The second statement follows at once from \ref{lem:0} and
\ref{prop:adfaith}, while the last  
assertion of the theorem follows immediately from the second one.    
 \end{proof}

\begin{corollary}\label{cor:functor}
The assignment
\begin{equation*}
  (E\stackrel{\pi}{\to}F)\longmapsto
  (E,E_0^r,\pr{\,}{\,}_r)\stackrel{(\pi,\pi_0^r)}{\longmapsto}(F,F_0^r,\pr{\,}{\,}_r)
\end{equation*}
defines a functor from the category of admissible $*$-trings into the
category of admissible and full basic triples. 
\end{corollary}

\begin{corollary}\label{cor:adfaith}
Let $(E,A,\pr{\ }{\ }_A)$ be a basic triple such that $E$ is
faithful as an $A$-module and $E$ is admissible as a $*$-tring. Then
there exists a unique homomorphism $\psi:E_0^r\to A$ such that
$\pr{x}{y}_r=\pr{x}{y}_A$, $\forall x,y\in E$. The 
homomorphism $\psi$ is injective, and it is an isomorphism if
$(E,A,\pr{\ }{\ }_A)$ is full.    
\end{corollary}
\begin{proof}
Let $(\,E,E_0^r,\pr{\ }{\ }_r\,)$ be the full and
admissible basic triple provided by Theorem~\ref{thm:*level}. The
identity map on $E$ is an injective homomorphism of $*$-trings, so 
by \ref{lem:0} there exists a unique homomorphism $\psi:E_0^r\to A$
such that $\pr{x}{y}_r=\pr{x}{y}_A$, $\forall x,y\in E$, which is
injective because $E$ is faithful as $E_0^r$-module. It is clear that
$\psi$ is also surjective when the given basic triple is full.
 \end{proof} 
\begin{corollary}\label{cor:adad}
Let $(E,A,\pr{\ }{\ }_A)$ be a full basic triple such that $E$ is 
faithful as a $A$-module. Then $A$ is admissible if
$E$ is admissible.
\end{corollary}
\begin{proof}
Just note that if $E$ is
admissible, then $E_0^r\cong A$ by \ref{cor:adfaith}, and $E_0^r$ is
admissible according to \ref{thm:*level}.
 \end{proof}
\begin{corollary}\label{cor:bobita}
Let $F$ be an ideal of the admissible $*$-tring $E$, $(E,E_0^r,\pr{\
}{\ }_E)$ and $(F,F_0^r,\pr{\ }{\ }_F)$ the full and admissible basic
triples associated, respectively, with $E$ and $F$ (given by Theorem \ref{thm:*level}). If $A:=\textrm{span}\{\pr{x}{y}_E:x,y\in F\}$, then
$A$ is a $*$-ideal of $E_0^r$, and the basic triples $(F,F_0^r,\pr{\ }{\ }_F)$ and $(F,A,\pr{\ }{\ }_E)$ are isomorphic.
\end{corollary}
\begin{proof}
The triple $(F,A,\pr{\ }{\ }_r)$ is admissible and full, with $F$ as
induced $*$-tring. Then $F$ is a faithful $A$-module by
\ref{prop:adfaith}. According to \ref{cor:adfaith}, there exists a
unique map $\psi:F_0^r\to A$ such that 
$(id,\psi)$ is a homomorphism from $(F,F_0^r,\pr{\ }{\ }_F)$ to
$(F,A,\pr{\ }{\ }_E)$, and $\psi$ is an isomorphism o $*$-algebras. It
follows that $(id,\psi^{-1})$ is the inverse homomorphism of
$(id,\psi)$.   
 \end{proof}
\par From now on if $F$ is an ideal in the admissible $*$-tring $E$, we
will think of $F_0^r$ as a $*$-ideal of $E_0^r$ via the identification
provided by \ref{cor:bobita}: 
\begin{equation}\label{eq:bobita}F_0^r\cong
  \textrm{span}\{\pr{x}{y}_E:x,y\in F\}.\end{equation}

For the next result recall that an ideal $F$ of the $*$-tring $E$ is
hermetic if and only if $E/F$ is admissible.

\begin{proposition}\label{prop:quotient0}
Let $\pi:E\to F$ be a homomorphism between the admissible $*$-trings
$E$ and $F$, such that $\ker\pi$ is hermetic. If $I_{\ker\pi}:=\{a\in
E_0^r: Ea\subseteq\ker\pi\}$, then:
\begin{equation*}
  (\ker\pi)_0^r\subseteq\ker\pi_0^r\subseteq I_{\ker\pi}
\end{equation*}
\end{proposition}
\begin{proof}
Taking into account \eqref{eq:bobita} above and the second part of
\ref{thm:*level}, the first inclusion is clear. The second inclusion
follows from the admissibility of $E/\ker\pi$ and \eqref{eq:ker} in 
Lemma~\ref{lem:0}.  
 \end{proof}

\begin{remark}\label{rem:iso1}
Suppose $F$ is an hermetic ideal of the admissible $*$-tring $E$. Let
$q:E\to E/F$ be the quotient map, $I_F:=\{a\in E_0^r: Ea\subseteq
F\}$, $p:E_0^r\to E_0^r/I_F$ the canonical projection and
$\overline{q_0^r}:E_0^r/I_F\to (E/F)_0^r$  
the isomorphism induced by $q_0^r$, so the following diagram commutes:  

\begin{equation*}
  \xymatrixcolsep{2ex}\xymatrixrowsep{2ex}\xymatrix{
    E_0^r\ar[rd]_{p}\ar@/^0pc/[urrd]^{q_0^r} 
    &&(E/F)_0^r\\
    &E_0^r/I_F\ar[ru]_{\overline{q_0^r}}&}
\end{equation*}

   Then:  
\begin{equation*}
  \overline{q_0^r}(p(\pr{x}{y}_E))
    =q_0^r(\pr{x}{y}_E))
    =\pr{q(x)}{q(y)}_{E/F},
    \quad\forall x,y\in~E.
\end{equation*}
Therefore the pair $((E/F)_0^r.\pr{\,}{\,}_{E/F})$ associated with
$E/F$ in Theorem~\ref{thm:*level} may be replaced by the pair
$(E_0^r/I_F, [\,,\,]_{E/F})$., where
$[q(x),q(y)]_{E/F}=p(\pr{x}{y}_E)$, $\forall x,y\in E$ and the action
of $E_0^r/I_F$ on $E/F_\gamma$ is given by $q(x)p(a)=q(xa)$, $\forall
x\in E$, $a\in A$.
\end{remark}

\begin{proposition}\label{prop:*(pi)0}
Let $\pi:E\to F$ be a \hm\ between admissible $*$-trings. Then: 
\be
 \item  $\pi$ is injective if and only if $\pi_0^r:E_0^r\to F_0^r$ is
   injective. 
 \item If $\pi$ is onto, or an isomorphism, then so is
   $\pi_0^r:E_0^r\to F_0^r$. 
\ee
\end{proposition}
\begin{proof}
Since the second statement is clear we prove only the first 
one. Now if $\pi_0^r$ is injective and  $x\in E$, the
admissibility of $E$ and $F$ implies that: 
\begin{equation*}
  \pi(x)=0\iff \pr{\pi(x)}{\pi(x)}_r=0\iff \pi_0^r(\pr{x}{x}_r)=0\iff x=0,
\end{equation*} so $\pi$ 
is injective as well. On the other hand the injectivity of $\pi$
implies that of $\pi_0^r$ by \ref{lem:0}. 
 \end{proof}

\section{Correspondence between \texorpdfstring{$C^*$}{C*}-seminorms.}

\subsection{\texorpdfstring{$C^*$}{C*}-seminorms.}\label{subsubsec:seminorms} 
\begin{definition}\label{defn:cnorm}
A $C^*$-seminorm on a $*$-tring $(E,\mu )$ is a seminorm such that: 
\be
\item $\norm{\mu (x,y,z)}\leq\norm{x}\,\norm{y}\,\norm{z}$, 
      $\forall x,y,z\in E$. 
\item\label{item:identity for norm of trings} $\norm{\mu (x,x,x)}=\norm{x}^3$, $\forall x\in E$.
\ee
If $\norm{\cdot}$ is a norm, we call it a $C^*$-norm, and we say that
$(E,\norm{\cdot})$ is a pre-$C^*$-ternary ring. If $(E,\norm{\cdot})$
is also a Banach space, we say that it is a $C^*$-ternary ring, or
just a \ct.   
\end{definition}
\par If $E$ is a \st, the set of $C^*$-seminorms on $E$ will be
denoted by $\msn(E)$, and $\mn(E)$ will denote the set of $C^*$-norms
on $E$. The set $\msn(E)$ is partially ordered by:
$\gamma_1\leq\gamma_2$ if $\gamma_1(x)\leq\gamma_2(x)$, $\forall x\in
E$.
\begin{definition}\label{defn:closable}
A $*$-tring $E$ will be called $C^*$-closable, or just closable, in
case $\mn(E)\neq\emptyset$. Similar terminology will be used for
$*$-algebras.  
\end{definition}

\par Observe that any $C^*$-closable $*$-tring is admissible. 

\par In the next proposition, whose easy proof is left to the reader,
we record some basic facts about $*$-trings. 
\begin{proposition}
Let $E$ be a $*$-tring. Then:
\be
 \item $N_\gamma:=\{x\in E:\gamma(x)=0\}$ is an hermetic ideal 
   of $E$, for all $\gamma\in\msn(E)$. 
 \item The intersection of hermetic subrings is also hermetic. 
 \item The quotient $E/N$ is admissible, where
   $N:=\cap\{N_\gamma:\gamma\in \msn(E)\}$ and $N_\gamma$ is as in
   1.  
 \item If $\msn(E)$ separates points of $E$, then $E$ is admissible. 
 \item If $\msn(E)$ separates points of $E$ and is bounded, then $E$
   is $C^*$-closable.  
\ee
\end{proposition}

\par If $H$ and $K$ are Hilbert spaces and $B(H,K)$ denotes the
corresponding space of bounded linear maps, a subspace $E$ of $B(H,K)$
closed under the ternary product $(R,S,T)\mapsto RS^*T\in E$, $\forall 
R,S,T\in E$, is a $*$-tring with that product.
In case $E$ is also closed it is called a ternary ring of operators
(TRO).
Note that if $(E,\mu)$ is a \ct, then $(E,-\mu)$ also is a \ct, called the
opposite of $(E,\mu)$ and denoted by $E^{\text{op}}$. The opposite
of a TRO is called anti-TRO.  
\par New $C^*$-ternary rings can be obtained by direct sums: if 
$(E,\norm{\cdot}_E,\mu_E)$ and $(F,\norm{\cdot}_F,\mu_F)$ are \cts,
then $(E\oplus
F,\max\{\norm{\cdot}_E,\norm{\cdot}_F\},\mu_E\oplus\mu_F)$ 
is a \ct. We denote it just by $E\oplus F$.      
\par Suppose that $E$ is a full right Hilbert $A$-module, and
define the ternary product on $E$: $\mu_E(x,y,z):=x\pr{y}{z}$. Then 
$(E,\mu_E)$ is a \ct with the norm $\norm{x}=\sqrt{\pr{x}{x}}$. Now,
if $F$ is a full right Hilbert $B$-module, then $E\oplus
F^{\text{op}}$ is also a \ct. This is the fundamental example of 
\ct, as shown by Zettl\label{zettl} in \cite[3.2]{z2} (see also
Corollay~\ref{cor:cbt} below).
\par Zettl also showed that there exist unique sub-\cts $E_+$ and 
$E_-$ of $E$ such that $E=E_+\bigoplus E_-$, and $E_+$ is isomorphic
to a TRO, while $E_-$ is isomorphic to an anti-TRO (see \cite{z2}). 
The decomposition above is called the \textit{fundamental
decomposition} of $E$.\label{decfund} 
\begin{definition}\label{defn:+}
We say that a \ct $E$ is positive (negative) if $E=E_+$ (respectively:
if $E=E_-$). 
\end{definition}  
\par If $E$ is a \ct, we define $E_p:=E_+\oplus
E_-^{\text{op}}$. Then $E_p$ is a positive \ct. 
\par Let $E^*$ be the reverse $*$-tring of (the *-tring) $E$. It is clear that a norm
on $E$ is a $C^*$-norm if and only if is a $C^*$-norm on
$E^*$. Moreover, $E$ is a (positive) \ct if and only if so is $E^*$.  

\subsection{From pre-\texorpdfstring{$C^*$}{C*}-trings to
  pre-\texorpdfstring{$C^*$}{C*}-algebras}\label{subsec:ctrings}  
In what follows we will examine an
intermediate situation between the $*$-algebraic context of
\ref{thm:*level} and the $C^*$-context originally considered by Zettl.

If $\alpha$ is a seminorm on the vector space $X$, then 
$N_\alpha:=\{x\in X:\, \alpha(x)=0\}$ is a closed subspace of $X$, so
$X/N_\alpha$ is a normed space with the norm $\bar{\alpha}$ induced by
$\alpha$: $\bar{\alpha}(x+N_\alpha)=\alpha(x)$. The completion
$(X_\alpha,\bar{\alpha})$ of $(X/N_\alpha,\bar{\alpha})$ will be
referred to as the \textit{Hausdorff completion} of the seminormed space
$(X,\alpha)$, and the map $x\mapsto x+N_\alpha$ will be called the
canonical map. 

\par In case $\gamma$ is a $C^*$-seminorm on the ternary ring $E$, 
then $E/N_\gamma$ is a
pre-$C^*$-tring with the induced norm $\bar{\gamma}$. Thus the
corresponding Hausdorff completion $E_\gamma$ of $E$ is a $C^*$-tring.    
\begin{proposition}\label{prop:semi*level}
Suppose $E$ is an admissible $*$-tring and $\ga\in\msn(E)$. Let  
$\ga^r:E_0^r\to [0,\infty)$ be the 
   operator seminorm on $E_0^r$, that is: 
   \begin{equation}\label{eq:defphi}
   \ga^r(a):=\sup\{\ga(xa):\,\ga(x)\leq 1\}.\end{equation} 
   Then $\ga^r\in\msn(E_0^r)$, and
   $\ga^r\in\mn(E_0^r)\iff\ga\in\mn(E)$. Moreover the following
   relations hold: 
\begin{gather}
\gamma(xa)\leq\gamma(x)\gamma^r(a), \forall x\in E, a\in E^r_0\label{eq:op.sem}\\
\gamma^r(\pr{x}{y}_r)\leq\gamma(x)\gamma(y), \forall x,y\in E\label{eq:cs}\\
\ga(x)^2=\ga^r(\pr{x}{x}_r), \forall x\in E \label{eq:phi.inj}
\end{gather}   
\end{proposition}
\begin{proof}
Given $a=\sum_{i=1}^n\pr{x_i}{y_i}\in E^r_0$ the linear map $x\mapsto xa$ is bounded because $\ga(xa)\leq \ga(x)\sum_{i=1}^n\ga(x_i)\ga(y_i).$
Then \eqref{eq:op.sem} and \eqref{eq:cs} follow immediately and Definition \ref{defn:cnorm} implies \eqref{eq:phi.inj}.
With $a\in E^0_r$ as before and $x\in E$ we have  
\begin{equation*}
  (xa,xa,xa)
    =\sum_{i=1}^n((x,x_i,y_i),xa,xa)
    =\sum_{i=1}^n(x,(xa,y_i,x_i),xa),
\end{equation*} so
\begin{equation*}
  \ga(xa)^3
    =\ga(x,xaa^*,xa)
    \leq\ga^r(aa^*)\,\ga^r(a)\,\ga(x)^3, 
\end{equation*}
from where it follows that $\ga^r(a)^2\leq\ga^r(aa^*)\leq \ga^r(a)\ga^r(a^*)$.   
From the computations above is clear that
$\ga^r\in\mn(E_0^r)\iff\ga\in\mn(E).$
In particular $E_0^r$ is a $C^*$-closable algebra whenever $E$ is a
  $C^*$-closable tring.
 \end{proof}

\begin{definition}\label{defn:cbt}
Suppose $(E,A,\pr{\,}{\,}_A)$ is a basic triple such that $(E,\gamma)$
is a $C^*$-tring and a Banach module over the $C^*$-algebra
$(A,\alpha)$, and that $\pr{\,}{\,}_A:E\times E\to A$ is
continuous. Then the triple is said to be a $C^*$-basic triple. We say
that it is full if the ideal $\textrm{span}\{\pr{x}{y}_A:\, x,y\in
E\}$ of $A$ is dense in $A$.  
\end{definition}
\par The next two results will be useful for studying the relation
between a $C^*$-basic triple $(E,A,\pr{\,}{\,}_A)$ and the basic
triple $(E,E_0^r,\pr{\,}{\,}_r)$. What we will show first, in
\ref{prop:cbt}, 
is that $(E,E_0^r,\pr{\,}{\,}_r)$ can be embedded in
$(E,A,\pr{\,}{\,}_A)$.  

\begin{proposition}\label{prop:extbanach}
Let $A$ be a Banach $*$-algebra and $I$ a $*$-ideal of $A$, not
necessarily closed. Then any $C^*$-seminorm on $I$ can be extended to
a $C^*$-seminorm on $A$. If $I$ is dense, such extension is unique.
\end{proposition}
\begin{proof}
Consider $\alpha\in\msn(I)$, $\alpha\neq0 $. Let $I_\alpha$ be the
Hausdorff completion of $(I,\alpha)$, $p:I\to I_\alpha$ the canonical
map, and let $\pi:I_\alpha\to B(H)$ be a faithful representation. Now,
according to \cite[VI-19.11]{fd}, the representation $\pi p:I\to B(H)$
can be extended to a representation $\rho$ of $A$. Then
$a\mapsto\norm{\rho(a)}$ defines a $C^*$-seminorm on $A$ that extends
$\alpha$. Note that the continuity of $\rho$ implies the continuity of
$\alpha$, from which the uniqueness of the extension follows in case
$I$ is dense in $A$. 
 \end{proof}

\begin{corollary}\label{cor:fell}
Let $I$ be a $*$-ideal of the $C^*$-algebra $A$. Then the unique
$C^*$-norm one can define in $I$ is the restriction to $I$ of the norm
of $A$. 
\end{corollary}

\begin{proposition}\label{prop:cbt}
Let $(E,A,\pr{\,}{\,}_A)$ be a full $C^*$-basic triple, and $\gamma$
and $\alpha$ the corresponding norms of $E$ and $A$. Then
$(A,\alpha)$ is the completion of $(E_0^r,\gamma^r)$, and
$\pr{\,}{\,}_A$ is the continuous extension of $\pr{\,}{\,}_r$.   
\end{proposition}
\begin{proof}
Note that $E$ is admissible for it is a $C^*$-tring. On the other hand
$E$ is a faithful $A$-module: if $a\in A$ is such that $xa=0\,\forall
x\in E$, then $\pr{x}{y}_Aa=0\, \forall x,y\in E$, so it follows that
$ba=0$ for every $b$ in the dense ideal $\textrm{span}\{\pr{x}{y}_A:\,
x,y\in E\}$ of $A$, which implies $a=0$. 
Thus there exists, by \ref{prop:adfaith}, a unique homomorphism
$\psi:E_0^r\to A$ such that 
$\psi(\pr{x}{y}_r)=\pr{x}{y}_A$, $\forall x,y\in E$. Besides $\psi$ is
injective and $\psi(E_0^r)=\textrm{span}\{\pr{x}{y}_A:\, x,y\in E\}$
(thus we may suppose $E_0^r$ is a dense ideal of $A$). Now
\ref{cor:fell} implies $\gamma_0^r$ is the restriction of
$\alpha$ to $\psi(E_0^r)$ and, since the latter is dense in $A$, we
conclude that $A$ is the completion of $E_0^r$.    
 \end{proof}
\par As a consequence we obtain the following result, due to H.~Zettl:   

\begin{corollary}[\textit{cf.} \mbox{\cite[Proposition 3.2]{z2}}]\label{cor:cbt}
Let $(E,\gamma)$ be a $C^*$-tring and $E^r$ the completion of $E_0^r$
with respect to $\gamma^r$. Then $(E,E^r,\pr{\,}{\,}_r))$ is, up to
isomorphism, the unique full $C^*$- basic triple whose first component
is $E$.   
\end{corollary}

\begin{proposition}\label{prop:cbt2}
Let $\pi:E_1\to E_2$ be a homomorphism of $*$-trings between the
$C^*$-trings $E_1$ and $E_2$. Then there exists a
unique \hm $\pi^r:E^r\to F^r$ such that
$\pi^r(\pr{x}{y}_E)=\pr{\pi(x)}{\pi(y)}_F$, $\forall x,y\in E$, and
$\pi(xa)=\pi(x)\pi^r(a)$ $\forall x\in E$, $a\in E^r$. Consequently
$\pi$ is always contractive, and is isometric if and only if it is
injective. 
\end{proposition}
\begin{proof}
It is clear that, if the \hm $\pi^r$ exists, it must be an extension
of $\pi_0^r:E_0^r\to F_0^r$. Let $\rho:F^r\to B(H)$ be a faithful
representation. Then $\rho\pi_0^r$ is a representation of
$E_0^r$. Now, since $(E,E^r,\pr{\,}{\,}_r)$ is a $C^*$-triple, $E_0^r$
is a $*$-ideal in $E^r$. Therefore $\rho\pi_0^r$ can be uniquely
extended to a representation $\bar{\rho}:E^r\to B(H)$
(\cite[VI.19.11]{fd}). Since $\rho(F^r)$ is closed and
$\bar{\rho}(E^r)$ is a subset of the closure of $\rho\pi_0^r(E_0^r)$,
we have $\bar{\rho}(E^r)\subseteq\rho(F^r)$. Then take 
$\pi^r:=\rho^{-1}\bar{\rho}$. Note that
$\norm{\pi(x)}^2=\norm{\pi^r(\pr{x}{x})}\leq
\norm{\pr{x}{x}}=\norm{x}^2$, with equality if $\pi^r$ is
injective. This shows that $\pi$ is contractive. Finally, if $\pi$ is
injective, so is $\pi_0^r$ and, as in the proof of \ref{cor:fell},
this implies that $\pi^r$ also is injective, thus an isometry.        
 \end{proof}

\begin{corollary}[\textit{cf.} \mbox{\cite{env}[Proposition 4.1]}]\label{cor:functor2}
The assignment 
\begin{equation*}
  (E\stackrel{\pi}{\to}F)\longmapsto
    (E,E^r,\pr{\,}{\,}_r)\stackrel{(\pi,\pi^r)}{\longmapsto}(F,F^r,\pr{\,}{\,}_r)
\end{equation*}
defines a functor from the category of $C^*$-trings to the
category of full $C^*$-basic triples. 
\end{corollary}

\par It follows from Proposition~\ref{prop:semi*level} that any
$C^*$-seminorm on $E_0^r$ induced by a $C^*$-seminorm on $E$ by means of
\eqref{eq:defphi} must satisfy the Cauchy-Schwarz
condition~\eqref{eq:cs}. So it is natural to restrict our attention to
the following subsets of $C^*$-seminorms on $E_0^r$: 
\begin{equation*}
  \msn_{cs}(E_0^r)
    :=\{\alpha\in\msn(E_0^r):\alpha(\pr{x}{y}_r)^2\leq\alpha(\pr{x}{x}_r)\,\alpha(\pr{y}{y}_r)\}
\end{equation*} 
\begin{equation*}
  \mn_{cs}(E_0^r):=\msn_{cs}(E_0^r)\cap \mn(E_0^r).
\end{equation*}

\par In fact it will be convenient to place ourselves in a slightly
more general setting: 

\begin{definition}
Let $(E,A,\pr{\,}{\,})$ be a basic triple.
We define
\begin{equation*}
  \msn_{cs}^{\pr{}{}}(A)
    :=\{\alpha\in\msn(A):\ \alpha(\pr{x}{y})^2\leq\alpha(\pr{x}{x})\alpha(\pr{y}{y}),\, 
\forall x,y\in E\}.
\end{equation*}
\end{definition}

\begin{proposition}\label{prop:semi*level2a}
 Let $(E,A,\pr{\,}{\,})$ be a basic triple, and consider $E$ with
 the $*$-tring structure induced by $\pr{\,}{\,}$. Given $\alpha\in
 \msn_{cs}^{\pr{}{}}(A)$, let $\tilde{\alpha}:E\to 
 [0,\infty)$ be defined by:
\begin{equation}\label{eq:defpsia}
\tilde{\alpha}(x):=\alpha(\pr{x}{x})^{1/2}
\end{equation}
Then 
\be
 \item $\tilde{\alpha}(xa)\leq\tilde{\alpha}(x)\alpha(a)$.
 \item $\tilde{\alpha}\in\msn(E)$
 \item If $E$ is a faithful $A$-module and $\tilde{\alpha}\in\mn(E)$,
   then $\alpha\in\mn_{cs}^{\pr{}{}}(A)$.
 \item If $\alpha\in\mn_{cs}^{\pr{}{}}(A)$ and $\pr{x}{x}=0$ implies
   $x=0$, then $\tilde{\alpha}\in\mn(E)$ 
\ee 
\end{proposition}
\begin{proof}
Since the Cauchy-Schwarz inequality \eqref{eq:cs} holds for $\alpha$,
it follows as usual that $\tilde{\alpha}$ satisfies the 
triangular inequality and, since homogeneity is obvious,
$\tilde{\alpha}$ is a seminorm on $E$. On the other hand, since $\alpha$ is a
$C^*$-seminorm and satisfies \eqref{eq:cs} we have, for all $x,y,z\in
E$,  $a\in A$:
\begin{gather*}
\tilde{\alpha}(xa)
=\alpha(a^*\pr{x}{x}a)^{1/2}\leq\alpha(a)\tilde{\alpha}(x)\\ 
\tilde{\alpha}((x,y,z))=
\tilde{\alpha}(x\pr{y}{z})
\leq\tilde{\alpha}(x)\alpha(\pr{y}{z})
\leq \tilde{\alpha}(x)\tilde{\alpha}(y)\tilde{\alpha}(z)\\
\tilde{\alpha}((x,x,x))=
\alpha(\pr{x}{x}^3)^{1/2}=
\alpha(\pr{x}{x})^{3/2}=
\tilde{\alpha}(x)^3,
\end{gather*}
so $\tilde{\alpha}$ is a $C^*$-seminorm on $E$. The first of the above
inequalities implies that $\alpha$ is a norm whenever $\tilde{\alpha}$
so is and $E$ is a faithful $A$-module. Finally, if $\alpha$
is a norm, it follows directly from \eqref{eq:defpsia} that
$\tilde{\alpha}$ also is a norm when the condition $\pr{x}{x}=0$
implies $x=0$. 
 \end{proof}

\begin{corollary}\label{cor:fa}
If $E$ is an admissible $*$-tring and $\gamma\in\msn(E)$,
$\alpha\in\msn_{cs}(E_0^r)$, then $\widetilde{\gamma^r}=\gamma$ and
$\tilde{\alpha}^r\leq\alpha$. 
\end{corollary}
\begin{proof}
The first statement follows immediately from \eqref{eq:phi.inj} and
\eqref{eq:defpsia}. As for the second one we have
$\tilde{\alpha}^r(a)=\sup\{\tilde{\alpha}(xa):\tilde{\alpha}(x)\leq
1\}\leq \alpha(a)$ by 1. of \ref{prop:semi*level2a}.  
 \end{proof}
\begin{corollary}\label{cor:quotient0}
 Let $(E,A,\pr{\,}{\,})$ be a full basic triple, and $\alpha\in
 \msn_{cs}^{\pr{}{}}(A)$. If $\tilde{\alpha}\in\msn(E)$ is given by
 \eqref{eq:defpsia}, then 
 $I_{N_{\tilde{\alpha}}}=N_\alpha$, where
 $I_{N_{\tilde{\alpha}}}:=\{a\in A:Ea\subseteq N_{\tilde{\alpha}}\}$.      
\end{corollary}
\begin{proof}
The inclusion $N_\alpha\subseteq I_{N_{\tilde{\alpha}}}$ is clear
because $\tilde{\alpha}(xa)\leq\tilde{\alpha}(x)\alpha(a)$, $\forall
x\in E$, $a\in A$. Conversely, suppose that $a\in A$ is such that
$\tilde{\alpha}(xa)=0$, $\forall x\in E$. Then
$\alpha(a^*\pr{x}{y}a)=\alpha(\pr{xa}{ya})\leq
\tilde{\alpha}(xa)\tilde{\alpha}(ya)=0$, $\forall x,y\in E$. Now,
since the triple is 
full, we can write $aa^*=\sum_j\pr{x_j}{y_j}$, for certain $x_j,y_j\in
E$, so we have:
\begin{equation*}
  0\leq \alpha(a)^4
    =\alpha(a^*a)^2=\alpha(a^*aa^*a)
    =\alpha(a^*\sum_j\pr{x_j}{y_j}a)
    \leq\sum_j\alpha(a^*\pr{x_j}{y_j}a)
    =0, 
\end{equation*}    
hence $a\in N_\alpha$. 
 \end{proof}

\begin{proposition}\label{prop:quotzettl}
 Let $(E,A,\pr{\,}{\,})$ be a full basic triple, and $\alpha\in
 \msn_{cs}^{\pr{}{}}(A)$. Let $\gamma:=\tilde{\alpha}\in\msn(E)$,
 $\tilde{\alpha}$  given by \eqref{eq:defpsia}. Then $E_\gamma$ is a
 $C^*$-tring,  
 $(E_{\gamma}^r,\bar{\gamma}^r)=(A_\alpha,\bar{\alpha})$ and
 $\tilde{\alpha}^r=\alpha$.  
\end{proposition}
\begin{proof}
Denote by $q:E\to E/N_\gamma\subseteq E_\gamma$ and $p:A\to
A/N_\alpha\subseteq A_\alpha$ the corresponding canonical
maps. 
We define $E/N_\gamma\times A/N_\alpha\to E/N_\gamma$ and
$[\,,\,]:E/N_\gamma\times E/N_\gamma\to A/N_\alpha$ such that
$q(x)p(a):=q(xa)$ and  
$[q(x),q(y)]:=p(\pr{x}{y})$ respectively. 
Let us see that these operations are continuous in the norms
$\bar{\gamma}$ and $\bar{\alpha}$. 
The action of $A/N_{\alpha}$ on $E/N_\gamma$ is continuous, for if $x,y\in
E$ and $a\in A$: 
\begin{equation*}
  \bar{\gamma}(q(x)p(a))
    =\bar{\gamma}(q(xa))
    =\gamma(xa)
    \leq\gamma(x)\alpha(a)
    =\bar{\gamma}(q(x))\bar{\alpha}(p(a))
\end{equation*}
And the sesquilinear map $[\,,\,]_{E/N_\gamma}$ also is continuous,
because:  
\begin{equation*}
  \bar{\alpha}([q(x),q(y)]_{E/F_\gamma})
    =\bar{\alpha}(p(\pr{x}{y}_{E}))
    =\alpha(\pr{x}{y}_{E})
    \leq \gamma(x)\gamma(y)
    =\bar{\gamma}(q(x))\bar{\gamma}(q(y)).
\end{equation*}
Therefore these operations extend to continuous maps $E_\gamma\times 
A_\alpha\to E_\gamma$ and $[\,,\,]:E_\gamma\times E_\gamma\to
A_\alpha$, so we obtain a full $C^*$-basic triple
$(E_\gamma,A_\alpha,[\,,\,])$. Therefore
$(A_\alpha,\alpha)=(E_\gamma^r,\bar{\gamma}^r)$ by
\ref{prop:cbt}. As for the last assertion, we have to prove that
$\gamma^r=\alpha$ or, equivalently, that
$\bar{\gamma^r}=\bar{\alpha}$. So it is enough to show that
$\gamma^r=\bar{\gamma}^rp$. But, if $a\in A$: 
\begin{equation*}
  \bar{\gamma}^r(p(a))
    =\sup\{\bar{\gamma}(q(x)p(a)):\bar{\gamma}(q(x))\leq 1\}
    =\sup\{\bar{\gamma}(q(xa)):\gamma(x)\leq 1\}
    =\gamma^r(a). 
\end{equation*}  
 \end{proof}

\par Propositions~\ref{prop:semi*level} and \ref{prop:semi*level2a}
allow us to define maps $\Phi_r:\msn(E)\to\msn_{cs}(E_0^r)$ and
$\Psi_r:\msn_{cs}(E_0^r)\to \msn(E)$ such that
$\Phi_r(\gamma)=\gamma^r$, given by \eqref{eq:defphi}, and
$\Psi_r(\alpha)=\tilde{\alpha}$, given by \eqref{eq:defpsia}. We want
to show that in fact $\Phi_r$ and $\Psi_r$ are mutually inverse maps
that preserve the order.    

\begin{theorem}\label{thm:semi*main}
Let $E$ be an admissible $*$-tring. Then the maps
$\Phi_r:\msn(E)\to\msn_{cs}(E_0^r)$ and
$\Psi_r:\msn_{cs}(E_0^r)\to\msn(E)$ are mutually inverse 
isomorphisms of lattices. Moreover
$\Phi_r(\mn(E))=\mn_{cs}(E_0^r)$ and $\Psi_r(\mn_{cs}(E_0^r))=\mn(E)$.     
\end{theorem}
\begin{proof}
By Corollary~\ref{cor:fa} we have $\Psi_r\Phi_r=Id_{\msn(E)}$, and
Proposition~\ref{prop:quotzettl} shows that
$\Phi_r\Psi_r=Id_{\msn_{cs}(E_0^r)}$, so the maps 
$\Phi_r$ and $\Psi_r$ are mutually inverse. Besides, it follows from 
\ref{prop:semi*level} that $\Phi_r(\gamma)$ is a norm if and only if
so is $\gamma$. On the other hand is clear that $\Psi_r$ preserves the
order, thus it remains to be shown that $\Phi_r$ also preserves the
order. To this end consider $\gamma_1\leq\gamma_2$ in $\msn(E)$. Since
$id:(E,\gamma_2)\to (E,\gamma_1)$ is continuous, it induces a
\hm $\pi:E_{\gamma_2}\to E_{\gamma_1}$, which in turn induces,  
according with Proposition~\ref{prop:cbt2}, a \hm
$\pi^r:E_{\gamma_2}^r\to E_{\gamma_1}^r$, which is necessarily
contractive. 
Thus if $a\in E_0^r$, we have:
\begin{equation*}
  \gamma_1^r(a)
    =\bar{\gamma_1^r}(\pi^r(a+N_{\gamma_2^r}))
    \leq\bar{\gamma_2^r}(a+N_{\gamma_1^r})
    =\gamma_2^r(a), 
\end{equation*}
which shows that $\gamma_1^r\leq\gamma_2^r$.
 \end{proof}
\par All we have done to the right side can be done also to the left
side. For example, every admissible $*$-tring $E$ induces a (left)
admissible and full basic triple $(E,E_0^l,\pr{\,}{\,}_l)$, we have an
isomorphism of posets $\Phi_l:\msn(E)\to\msn_{cs}(E_0^l)$ with inverse
$\Psi_l:\msn_{cs}(E_0^l)\to \msn(E)$, given by 
$\Phi_l(\gamma)=\gamma^l$ and $\Psi(\alpha)=\tilde{\alpha}$, where
$\gamma^l(a):=\sup\{\gamma(ax):\gamma(x)\leq 1\}$ and
$\tilde{\alpha}(x):=\alpha(\pr{x}{x}_l)^{1/2}$, etc. Then we obtain
the following consequences: 
 
\begin{corollary}\label{cor:semi*main}
Let $E$ be an admissible $*$-tring. Then
$\Phi_r\Psi_l:\msn_{cs}(E_0^l)\to\msn_{cs}(E_0^r)$ is an
isomorphism of lattices such that
$\Phi_r\Psi_l(\mn_{cs}(E_0^l))=\mn_{cs}(E_0^r)$. The inverse of
$\Phi_r\Psi_l$ is $\Phi_l\Psi_r$.  
\end{corollary}
\par As mentioned at the end of \ref{subsubsec:seminorms} in
\cite{z2}[Theorem~3.1], Zettl proved that any $C^*$-tring is of the
form $E=E_+\oplus E_-$, where $E_+$ and $E_-^{op}$ are isomorphic to a
TRO. In fact we have $E_{+}:=\{ x\in E: \pr{x}{x}_r \textrm{ is
  positive}\}$, $E_{-}:=\{ x\in E: -\pr{x}{x}_r \textrm{ is
  positive}\}$, and $E_+$ and $E_-$ are ideals of $E$ such that
$\pr{E_{+}}{E_{-}}=0$. If $E_p:=E_+\oplus E_-^{op}$, we will have that
$E_p^r=E^r$ and $E_p^l=E^l$, and now $E_p$ is a Morita-Rieffel
equivalence between $E^l$ and $E^r$. Thus we have:  
  
\begin{corollary}\label{cor:preeq}
Let $E$ be an admissible $*$-tring and $\gamma\in\msn(E)$. Then
$E_{\gamma}^l$ and $E_{\gamma}^r$ are Morita-Rieffel equivalent \css.     
\end{corollary}

\par In general we will have to deal with algebras that strictly
contain $E_0^r$, but whose $C^*$-seminorms are essentially the same, as the
following results show.  

\begin{proposition}\label{prop:wbijnorms}
Let $I$ be a selfadjoint ideal of a $*$-algebra $A$, and suppose that
$\al\in\msn(I)$. Let $\al':A\to [0,\infty]$ be given by
$\al'(a):=\sup\{\al(ax):x\in I,\al(x)\leq 1\}$. For every $a\in A$
consider $L_a:I\to I$, such that $L_a(x)=ax$, $\forall x\in I$. Then
the following statements are equivalent:
\be \item $\al'(a)<\infty$, $\forall a\in A$.
    \item $L_a$ is bounded, $\forall a\in A$. 
    \item $\al$ can be extended to a $C^*$-seminorm on $A$.
\ee
Suppose that one of the conditions above holds true. Then:
\be[(a)]
 \item $\al'$ is a $C^*$-seminorm on $A$, and $\al'\leq\beta$ for
       every $\beta\in \msn(A)$ that extends~$\al$.
 \item If $\al$ is a norm, then $\al'$ is a norm if and only if 
       $\text{Ann}_A(I)=0$, where $\text{Ann}_A(I):=\{a\in A:\,ax=0,
       \forall x\in I\}$.  
\ee  
\end{proposition}
\begin{proof}
Since $\norm{L_a}=\al'(a)$, we have that conditions 1. and 2. are
equivalent. It is also clear that $3.\Rightarrow 1.$ Suppose now that
$\al'(a)<\infty$, $\forall a\in A$. Let show that $\al'$ is a
$C^*$-seminorm on $A$ that extends $\al$. It is easy to check that
$\al'(ab)\leq\al'(a)\al'(b)$, $\forall a,b\in A$. Moreover:
\begin{align*}
\al'(a^*a)&=\sup\{\al(a^*ax):\,x\in I, \al(x)\leq 1\}
            \geq\sup\{\al(x^*a^*ax):\,x\in I, \al(x)\leq 1\}\\
          &  \geq \sup\{\al(ax)^2:\,x\in I, \al(x)\leq 1\}
            =\al'(a)^2.
\end{align*}
Therefore $\al'\in\msn(A)$. 
The fact that $\al'$ extends $\al$, as well as 
assertion (a), are consequences of the fact that for every $C^*$-seminorm
$\beta$ on $A$ one has that $\beta(a)=\sup\{\beta(ab):\,\beta(b)\leq
1\}$. Finally, suppose that $\al$ is a norm on $I$. Then
$\al'(a)=0\iff\al(ax)=0$, $\forall x\in I$, that is $\al'(a)=0\iff
a\in\text{Ann}_A(I)$.    
 \end{proof}

\begin{theorem}\label{thm:main0}
Let $(E,A,\pr{\,}{\,})$ be an admissible basic triple, 
with $E$ a faithful $A$-module, and admissible as $*$-tring. Suppose
that any $C^*$-seminorm on $E_0^r$ can be extended in a unique way to
a $C^*$-seminorm on $A$ (recall 
Corollary~\ref{cor:adfaith}). Then the lattices $\msn(E)$ and
$\msn_{cs}^{\pr{\,}{\,}}(A)$ are isomorphic. If in addition
$\textrm{Ann}_A(E_0^r)=0$, the posets $\mn(E)$ and
$\mn_{cs}^{\pr{\,}{\,}}(A)$ are isomorphic as well.      
\end{theorem}
\begin{proof}
Since any $C^*$-seminorm on $E_0^r$ can be uniquely extended to a
$C^*$-seminorm on $A$, we are allowed to identify $\msn(A)$ and 
$\msn(E_0^r)$ as lattices, and it is clear that this
yields also an identification between $\msn_{cs}^{\pr{\,}{\,}}(A)$ and
$\msn_{cs}(E_0^r)$, and the latter is isomorphic to $\msn (E)$ by
\ref{thm:semi*main}. If moreover $\textrm{Ann}_A(E_0^r)=0$, the same
argument applies to $\mn(E)$ and $\mn_{cs}(A)$.  
 \end{proof}

\par In case $A$ is a Banach $*$-algebra,  any
$C^*$-seminorm  on a $*$-ideal can be extended to a $C^*$-seminorm
defined on the whole algebra. Moreover we have: 

\begin{proposition}\label{prop:extbanach2}
Let $A$ be an admissible Banach $*$-algebra and $I$ a dense $*$-ideal
of $A$, not necessarily closed. Then any $C^*$-norm on $I$ can be
uniquely extended to a $C^*$-norm on $A$.    
\end{proposition}
\begin{proof}
Let $\alpha\in\mn(I)$. By \ref{prop:extbanach} $\alpha$ has a unique
extension to a $C^*$-seminorm on $A$, and by \ref{prop:wbijnorms} this
extension must be $\alpha'$ such that
$\alpha'(a)=\sup\{\alpha(ax):x\in I,\alpha(x)\leq 1\}$. Suppose
$a\in\textrm{Ann}_A(I)$. Then $aa^*=0$, because $I$ is dense in $A$
and $ax=0$, $\forall x\in I$. Thus $a=0$ for $A$ is admissible. Then
$\alpha'$ is a norm by \ref{prop:wbijnorms}.   
 \end{proof}  

\begin{corollary}\label{cor:0}
Let $(E,A,\pr{\,}{\,}_E)$ be an admissible basic triple with $A$ a
Banach $*$-algebra and $E$ a faithful $A$-module. Suppose in 
addition that $E$ is an admissible $*$-tring such that $E_0^r$ is a
dense ideal of $A$ (recall
Corollary~\ref{cor:adfaith}). Then the
lattices $\msn(E)$ and 
$\msn_{cs}(A)$ are isomorphic, as well as the partially ordered sets 
$\mn(E)$ and $\mn_{cs}(A)$.    
\end{corollary}
\begin{proof}
Just combine Theorem~\ref{thm:main0} with
Proposition~\ref{prop:extbanach} and
Proposition~\ref{prop:extbanach2}. 
 \end{proof}
\begin{corollary}\label{cor:00}
Let $(E,A,\pr{\,}{\,}_A)$ and $(E,B,\pr{\,}{\,}_B)$ be respectively left
and right admissible basic triples, with $A$ and $B$ Banach
$*$-algebras such that $E$ 
is an $(A-B)$-bimodule with the given structure, and
$\pr{x}{y}_Az=x\pr{y}{z}_B$, $\forall x,y,z\in E$. If $E$ is faithful as
a left $A$-module and as a right $B$-module, and $E_0^l$ and $E_0^r$
are dense in $A$ and $B$ respectively, then there is an 
isomorphism of lattices between 
$\msn_{cs}^{\pr{}{}_A}(A)$
and
$\msn_{cs}^{\pr{}{}_B}(B)$,
that restricts to an isomorphism between the posets 
$\mn_{cs}^{\pr{}{}_A}(A)$
and
$\mn_{cs}^{\pr{}{}_B}(B)$.
\end{corollary}
\subsection{Positive modules}\label{subsec:pm} 
\par In general is not a simple task to decide if a given
$C^*$-seminorm satisfies the Cauchy-Schwarz property with respect to a
certain sesquilinear map. However this is always the case for the
positive modules we introduce next. 

\par Let $\al$ be a $C^*$-seminorm on the $*$-algebra $A$, and let 
$p_\al:A\to A_\al$ be the canonical map, where $A_\al$ is the
Hausdorff completion of $A$. If $\Lambda\subseteq\msn(A)$, then
$A_\Lambda^+:=\cap_{\alpha\in\Lambda}p_\alpha^{-1}(A_\alpha^+)$ is a
cone. When $\Lambda=\msn(A)$, we write $A^+$ instead of
$A_\Lambda^+$. Therefore $A^+$ is the set of elements of $A$ that are
positive in any $C^*$-Hausdorff completion of $A$. Of course the map
$\Lambda\mapsto A_\Lambda^+$ is order reversing.   
\begin{definition}\label{defn:pos*alg}
 Given $\Lambda\subseteq\msn(A)$, we say that $a\in A$ is positive in
 $(A,\Lambda)$, or that it is $\Lambda$-positive, if $a\in
 A_\Lambda^+$. The elements of $A^+$ are just called the positive 
 elements of~$A$.    
\end{definition}
\par It is clear that $A^+$ contains the cone
$C_A:=\{\sum_{i,j=1}^na_i^*a_j:\, n\in\nt, a_i\in A, i=1,\ldots n\}$,
and that $p_\al(C_A)$ is dense in $A_\al^+$, $\forall
\al\in\msn(A)$. Also note that if $\phi:A\to B$ is a homomorphism
between $*$-algebras, then $\phi(A^+)\subseteq B^+$ and
$\phi(C_A)\subseteq C_B$.    
\par If $\msn(A)$ is bounded, with $\alpha:=\max\msn(A)$, then $a$ is
positive in $A$ if and only if $a$ is positive in $(A,\alpha)$. In
particular, if $A$ is a Banach $*$-algebra, then $a\in A^+$ if and
only if $\iota(a)\in C^*(A)^+$, where $\iota:A\to C^*(A)$ is the
natural map of $A$ into its $C^*$-enveloping algebra $C^*(A)$.  

\begin{lemma}\label{lem:pos}
Let $A$ be $C^*$-closable. Then
$A^+=\bigcap\{p_\al^{-1}(A_\al^+):\,\al\in\mn(A)\}$.  
\end{lemma}
\begin{proof}
Clearly we have that
$A^+\subseteq \bigcap\{p_\al^{-1}(A_\al^+):\,\al\in\mn(A)\}$. Let
$\beta\in\msn(A)$. Since the maximum of two $C^*$-seminorms is again a
$C^*$-seminorm, and since $A$ is $C^*$-closable, we may pick
$\beta'\in\mn(A)$ such 
that $\beta'\geq\beta$. Then the identity map on $A$ induces 
a \hm $\phi:A_{\beta'}\to A_\beta$, determined by $\phi(p_{\beta'}(a))=p_\beta(a)$,
$\forall a\in A$. If $a\in
\bigcap\{p_{\al}^{-1}(A_{\al}^+):\,\al\in\mn(A)\}$ then $p_{\beta'}(a)\in
A_{\beta'}^+$, and therefore $p_\beta(a)\in A_\beta^+$. This 
proves the converse inclusion.   
 \end{proof}
Once  we have a cone of positive elements on a $*$-algebra $A$, we are
able to define a notion similar to that of Hilbert module. 
\begin{definition}\label{defn:positivemodule}
Let $A$ be a $*$-algebra, $E$ a right $A$-module, and   
$\Lambda\subseteq\msn(A)$. We say that a map $\pr{\cdot}{\cdot}: E\times
E\to  A$ is a $\Lambda$-semi-pre-inner product on $E$ if:
\be
\item $\pr{x}{\la_1 y+\la_2 z}=\la_1\pr{x}{y}+\la_2\pr{x}{z}$, $\forall
      x,y,z\in E$, $\la_1 ,\la_2\in \C$.
\item $\pr{x}{ya}=\pr{x}{y}a$, $\forall x,y\in E$, $a\in A$.
\item $\pr{y}{x}=\pr{x}{y}^*$, $\forall x,y\in E$.
\item $\pr{x}{x}$ is $\Lambda$-positive,  $\forall x\in E$.  
\ee
The pair $(E,\pr{\,}{\,})$ is then called a right positive
$\Lambda$-module. In case $\Lambda=\msn(A)$ we say that
$(E,\pr{\,}{\,})$ is a right positive $A$-module. 
\par Similarly we define left semi-pre-inner-products
and left positive modules. 
\end{definition}  
\begin{definition}\label{defn:+*}
An admissible $*$-tring $E$ is right (left) positive if $(E,\pr{\,}{\,}_r)$ 
is a positive $E_0^r$-module (respectively: $(E,\pr{\,}{\,}_l)$ 
is a positive $E_0^l$-module). It is said positive if it is both left 
and right positive.      
\end{definition}
\par Observe that if $E$ is a $C^*$-tring, which is positive as an 
admissible $*$-tring, then it is obviously a positive
$C^*$-tring. Conversely, it is readily checked that any positive
$C^*$-tring is a positive admissible $*$-tring. 

\begin{proposition}\label{prop:tilde}
Let $(A,\alpha)$ be a $C^*$-seminormed algebra and $(E,\pr{\,}{\,})$ a
right positive $(A,\alpha)$-module. Let $\tilde{\al}:E\to [0,\infty)$
  be given by 
$\tilde{\al}(x)=\sqrt{\al(\pr{x}{x})}$, $\forall x\in E$. Consider $E$
as a $*$-tring with $(x,y,z):=x\pr{y}{z}$, $\forall x,y,z\in E$. Then:  
\begin{enumerate}
 \item We have $\alpha(a)\leq\alpha(b)$ whenever $a$ and $b-a$ are
   positive elements of $A$. 
 \item $\tilde{\al}(x)^2\pr{y}{y}-\pr{x}{y}^*\pr{x}{y}$ is positive in
       $(A,\al)$, and $\alpha(\pr{x}{y})\leq
   \tilde{\alpha}(x)\tilde{\alpha}(y)$, $\forall x,y\in E$
   (Cauchy-Schwarz).    
 \item $\alpha(\pr{x}{x})a^*a-a^*\pr{x}{x}a\geq 0$, $\forall x\in E$, 
   $a\in A$.  
 \item $\tilde{\al}(xa)\leq\tilde{\al}(x)\al(a)$, $\forall x\in E$,
   $a\in A$.   
 \item $\tilde{\al}\in\msn(E)$. 
\end{enumerate}
\end{proposition}
\begin{proof}
Let $p_\al:A\to A/I_\al=:A_\alpha$ be the natural map, where
$I_\al$ is the ideal $I_\alpha:=\{a\in A:\,\al(a)=0\}$, and let
$\bar{\al}$ be the quotient 
norm on $A_\alpha$. Now let $F:=\gen\{xb\in E:\,x\in E, b\in
I_\al\}$. Then $EI_\alpha\subseteq F$, so $E/F$ is an
$A/I_\al$-module. Moreover, $\pr{E}{F}\subseteq I_\alpha$ and
$\pr{F}{E}\subseteq I_\alpha$, so we can consider the map 
$[\,,\,]:E/F\times E/F\to A/I_\al$ given by 
$[q(x),q(y)]=p_\al(\pr{x}{y})$, which satisfies properties 1.--4. of
Definition~\ref{defn:positivemodule} above. If $a$ and $b-a$ are
positive in $A$, then $0\leq p_\alpha(a)\leq p_\alpha(b)$ in
$A_\alpha$, and therefore $\bar{\alpha}(p_\alpha(a))\leq
\bar{\alpha}(p_\alpha(b))$, that is $\alpha(a)\leq\alpha(b)$. This
proves 1. Now, the first part of the second statement follows from the 
proof of \cite[Proposition 1.1]{l}, since
$p_\al(\tilde{\al}(x)^2\pr{y}{y}-\pr{y}{x}\pr{x}{y})=\bar{\al}([q(x),q(x)])[q(y),q(y)]  
-[q(y),q(x)]\,[q(x),q(y)]$ in $A_\al$. The second part of~2. follows
from the first one and from 1. To see 3. just observe that
by applying $p_\alpha$ to the element
$\alpha(\pr{x}{x})a^*a-a^*\pr{x}{x}a$ of $A$ we get the positive
element   
$\bar{\alpha}([x,x])p_\alpha(a)^*p_\alpha(a)-p_\alpha(a)^*[x,x]p_\alpha(a)$
of $A_\alpha$. Assertion 4. easily follows from 1. and 3: by
3. we have $a^*\pr{x}{x}a\leq \tilde{\alpha}(x)^2a^*a$, then 
$\tilde{\alpha}(xa)^2=\alpha(\pr{xa}{xa})=\alpha(a^*\pr{x}{x}a)$, and
by 1. this is less or equal to
$\alpha(\tilde{\alpha}(x)^2a^*a)=\tilde{\alpha}(x)^2\alpha(a)^2$.    
It is clear that $\tilde{\al}(\la x)=|\la|\tilde{\al}(x)$, $\forall
x\in E$, $\la\in\C$, and from the Cauchy-Schwarz inequality just
proved it readily follows that 
$\tilde{\al}$ also satisfies the triangle inequality, so it is a
seminorm on $E$.   
Now, if $x,y,z\in E$:
$\tilde{\al}(x\pr{y}{z})^2
=\al(\pr{y}{z}^*\pr{x}{x}\pr{y}{z}).$ 
Thus, in the case $x=y=z$: 
\begin{equation*}
  \tilde{\al}((x,x,x))
    =\al(\pr{x}{x}^3)^{1/2}
    =\al(\pr{x}{x}^{1/2})^3
    =\tilde{\al}(x)^3.
\end{equation*} 
According to 3. we have $\pr{y}{z}^*\pr{x}{x}\pr{y}{z}\leq
\al(\pr{x}{x})\pr{y}{z}^*\pr{y}{z}$ in $(A,\al)$. From this fact,
together with 4. and the Cauchy-Schwarz inequality we conclude that    
\begin{equation*}
  \tilde{\al}(x\pr{y}{z})^2
    \leq \tilde{\al}(x)^2\al(\pr{y}{z})^2
    \leq (\tilde{\al}(x)\,\tilde{\al}(y)\,\tilde{\al}(z))^2
\end{equation*} 
so $\tilde{\al}$ is a $C^*$-seminorm on $E$.    
 \end{proof}
\begin{corollary}\label{cor:possc}
If $E$ is a right positive $*$-tring, then
$\msn_{cs}(E_0^r)=\msn(E_0^r)$,  and $\msn(E)\cong\msn(E_0^r)$ and
$\mn(E)\cong\mn(E_0^r)$ as ordered sets. 
\end{corollary}

\begin{proposition}
Let $E$ be an admissible $*$-tring and $\gamma\in\msn(E)$. If $E$ is a
right positive $(E_0^r,\gamma^r)$-module, then $E$ is also a left
positive $(E_0^l,\gamma^l)$-module. Therefore $E$ is right positive if
and only if is left positive.     
\end{proposition}
\begin{proof}
Let $E_\gamma$ be the Hausdorff completion of $(E,\gamma)$. Since
$E_\gamma$ is a right Hilbert module over $E_{\gamma}^r$, it 
turns out that $E_\gamma$ is a positive $C^*$-tring, and therefore a
left Hilbert module over $E_\gamma^l$, so $E$ is a left
positive $(E_0^l,\gamma^l)$-module.           
 \end{proof}
\begin{proposition}\label{prop:typ}
Let $B$ be an admissible Banach $*$-algebra and suppose $E$ is a right
closed ideal of $B$ such that $\gen\{x^*y:\, x,y\in E\}$ is dense in $B$. Let 
$A$ be the closure in $B$ of $\gen\{xy^*:\, x,y\in E\}$. If $xx^*$ is
positive in $A$, $\forall x\in E$, then the restriction
map $\varphi:\msn(B)\to\msn(A)$, $\beta\mapsto\beta|_A$, is a lattice   
isomorphism such that $\varphi(\mn(B))=\mn(A)$, and for each
$\beta\in\msn(B)$ the Hausdorff completion $B_\beta$ of $B$ is
Morita-Rieffel equivalent to the Hausdorff completion
$A_{\varphi(\beta)}$ of $A$. In particular,   
the corresponding enveloping $C^*$-algebras $C^*(B)$ and $C^*(A)$ of
$B$ and $A$ are Morita-Rieffel equivalent $C^*$-algebras.  
\end{proposition}
\begin{proof}
Let $\pr{\,}{\,}_B:E\times E\to B$ and $\pr{\,}{\,}_A:E\times E\to A$
be such that $\pr{x}{y}_B=x^*y$ and $\pr{x}{y}_A=xy^*$. Then $E$ is
both a positive $B$-module and a positive $A$-module. Since $B$ is
admissible, so are $E$ and $A$. Besides $E$ is a faithful $B$-module,
for if $xb=0$ $\forall x\in E$, then $\sum_{j}x_j^*y_j b=0$ $\forall
x_j,y_j\in E$, so $b^*b=0$, and this implies $b=0$ because $B$ is
admissible. Similarly, $E$ is a faithful $A$-module. It follows by
\ref{cor:adfaith} that we can identify $E_0^r$ with $\gen\{x^*y:\,
x,y\in E\}$ and $E_0^l$ with $\gen\{xy^*:\, x,y\in E\}$. Now the proof
ends with an invocation to Corollary~\ref{cor:00}   
 \end{proof}

\section{\texorpdfstring{$C^*$}{C*}-ternary rings}

\par As previously mentioned, Zettl found a unique decomposition
$E=E_+\bigoplus E_-$ of any $C^*$-tring $E$, $E_+$ being isomorphic to
a TRO and $E_-$ being isomorphic to an anti-TRO
(see the discussion preceding Corollary \ref{cor:preeq}). 
Of course, because of the uniqueness of the fundamental decomposition,
there is a left version of the stuation above: 
$E_+:=\{ x\in E:\pr{x}{x}_l\in E^l_+\}$, 
$E_-:=\{ x\in E:\pr{x}{x}_l\in -E^l_+\}$, $\pr{E_+}{E_-}_l=0$,
$E^l=E_+^l\oplus E_-^l$, and  $(E_+,-\pr{\cdot}{\cdot}_l)$ and
$(E_-,-\pr{\cdot}{\cdot}_l)$ are full left  
Hilbert $E_+^l$ and $E_-^l$ modules respectively. This way, $E$
is an $(E^l-E^r)$ Banach bimodule that satisfies
\begin{equation*}
  \pr{x}{y}_lz
    =\mu (x,y,z)
    =x\pr{y}{z}_r,\ \forall x,y,z\in E.
\end{equation*}  
\par If $E$ is a \ct, we define $E_p:=E_+\oplus
E_-^{\text{op}}$. Then $E_p$ is a positive \ct, and $E_p^r=E^r$,
$E_p^l=E^l$. Therefore $E_p$ is a $(E^l-E^r)$-imprimitivity bimodule,
so in particular $E^l$ and $E^r$ are Morita-Rieffel equivalent. Note also that
if $\phi:E\to F$ is a \hm of \cts, then 
$\phi(E_+)\subseteq F_+$ and $\phi(E_-)\subseteq
F_-$, because $\pr{\phi(x)}{\phi(x)}=\phi^r(\pr{x}{x})$. Therefore
$\phi:E_p\to F_p$ is also a \hm of \cts. Thus $E\mapsto E_p$ is a
functor.   
\par Let $E^*$ be the reverse $*$-tring of $E$. It is clear that a norm
on $E$ is a $C^*$-norm if and 
only if it is a $C^*$-norm on $E^*$. Moreover, $E$ is a (positive) \ct
if 
and only if so is $E^*$, and $E^l=(E^*)^r$, $E^r=(E^*)^l$. Note that
$E$ and $E^*$ are essentially the same object as \cts.  
Thus the properties of $E^r$ and $E^l$ deduced from properties of $E$ 
will be the same. 
\begin{definition}\label{defn:ideals}
By a left (right) ideal of the $C^*-$ternary ring 
$E$ we mean a closed subspace $F$ of $E$
such that $(E,E,F)\subseteq F$ (respectively: $(F,E,E)\subseteq
F$). An ideal of $E$ is both a left and a right ideal of $E$. We
denote by $L(E)$, $R(E)$, and $I(E)$ the families of left, right and
twosided ideals of $E$.
\end{definition}

Our definition of ideal, for a closed subspace $F$ of $E$, is
equivalent to the definition 
which just requires the condition $(E,F,E)\subset F$ to be satisfied. 
Note that $E_+$ and $E_-$ are ideals in every \ct $E$. Moreover,
since $E_+$ and $E_-$ are orthogonal, it easily follows that a
closed subspace $F$ of $E$ is an ideal of $E$ if and only if it is an 
ideal in $E_p$. Thus the ideal structures of $E$ and of $E_p$ are the
same.   
\par If $A$ is a \cs, we will denote by $I(A)$ and $H(A)$
respectively the families of (closed) twosided ideals and
hereditary $C^*$-subalgebras of 
$A$. 
\par As in the algebraic case, if $E$ is a \ct and $F$ is a sub-\ct of
$E$, then the subalgebra  
$\ov{\textrm{span}}\pr{F}{F}_r$ of $E^r$ may be taken to represent the 
$C^*$-algebra $F^r$. With this choice of $F^r$ we have the following
result: 
\begin{proposition}\label{prop:correspondence}
The map $L(E)\to H(E^r)$ given by $F\mapsto F^r$ is a 
bijection, with inverse given by $A\mapsto EA$. When restricted to
$I(E)$, the map $F\mapsto F^r$ is a bijection onto $I(E^r)$. Moreover,
all of these maps are lattice isomorphisms.
\end{proposition}
\begin{proof}
We prove that the map $L(E)\to H(E^r)$ is a bijection. Recalling that we
may replace $E$ by $E_p$ (which can be seen as a full right Hilbert
$E^r$-module), the rest of the proof follows from  
\cite[3.22]{rw}. If $A$ is a \scs of $E^r$:
$(E,E,EA)=E\pr{E}{EA}=E\pr{E}{E}A=(E,E,E)A\subseteq EA,$ 
so $EA$ is a left ideal in $E$.
Conversely, if $F$ is a left ideal in $E$:
\begin{equation*}
  \pr{F}{F}\pr{E}{E}\pr{F}{F}
      =\pr{\,E\pr{F}{F}\,}{\,E\pr{F}{F}\,}
      =\pr{(E,F,F)}{(E,F,F)}
      \subseteq\pr{F}{F}.
\end{equation*} 
Thus, taking the closed linear spans in both sides of the above
inclusion we have:  
$F^rE^rF^r=F^r$, which shows that $F^r$ is hereditary.
To see that the correspondences are mutually inverses, note that if 
$F$ is a \ct, then $F=FF^r$. On the other hand, if $A$ is a hereditary
\scs of $E^r$, then $EA=\ov{\gen}\pr{EA}{EA}_r
=\ov{\gen}A\pr{E}{E}_rA=AE^rA=A$.
 \end{proof}

\begin{corollary}\label{cor:lattice}
Let $\pi:E\to F$ be a \hm of $*$-trings, where $E$ and $F$ are
\cts. Then $(\ker \pi)^r=\ker(\pi^r)$.  
\end{corollary}
\begin{proof}
It is clear that $\ker\pi\supseteq E \ker\pi^r$, so
$(\ker\pi)^r\supseteq\ker\pi^r$. On the other hand 
$(\ker\pi)^r=\overline{\textrm{span}\{\pr{x}{y}_r:\,
  x,y\in\ker\pi\}}\subseteq\ker\pi^r$.  
 \end{proof}
\begin{remark}\label{rem:lattice}
By Proposition \ref{prop:*(pi)0} if $\pi:E\to F$ is a
  surjective \hm between \cts, then $\pi^r_0:E_0^r\to F_0^r$ is also
  surjective, so also is $\pi^r:E^r\to F^r$ for the image of $\pi^r$
  is closed. However the
converse is false: consider the Hilbert space inclusion
$\C\stackrel{\iota}{\inc}\C^2$; then $\iota$ is not onto, although
$\iota^r$ is the identity on $\C$.
\end{remark}
For a proof of the next result the reader is referred to
\cite[3.25]{rw}. 
\begin{proposition}\label{prop:quotient} Let $F$ be an ideal of a \ct $E$,
and consider the quotient $E/F$ with its natural structure of
$*$-tring. Then $E/F$ is a \ct with the quotient norm, and  
$(E/F)^r=E^r/F^r$.
\end{proposition}
\begin{corollary}\label{cor:quotient}
Let $E$ and $G$ be \cts, and $\pi:E\to G$ a \hm of $*$-trings. 
Consider $F=\ker(\pi)$, and let $p:E^r\to E^r/F^r$ be the  
quotient map. Then there exists a unique \hm of \css 
$\ov{\pi^r}:E^r/F^r\to G^r$ such that $\ov{\pi^r}p=\pi^r$.  
The \hm $\ov{\pi^r}$ is injective. 
In particular, if $\pi:E\to E/F$ is the quotient map, where $F$ 
is an ideal of $E$, then $\ov{\pi^r}:E^r/F^r\to (E/F)^r$ is a   
natural isomorphism.
\end{corollary}
\begin{proof}
Proposition \ref{prop:cbt2} provides a unique \hm of  
\css $\pi^r:E^r\to G^r$ such that 
$\pr{\pi(x)}{\pi(y)}=\pi^r(\pr{x}{y})$, $\forall x, y\in E$. The
existence and uniqueness of $\ov{\pi^r}$, as well as its
injetivity, follow now from the quotient universal property, together
with the fact that $\ker(\pi^r)=F^r$ by Corollary
\ref{cor:lattice}. Finally, if $F$ is an ideal of $E$, by Proposition  
\ref{prop:quotient} we have that $E/F$ is a \ct, and the
projection $\pi:E\to E/F$ is a \hm of $*$-trings.   
 \end{proof}

\begin{corollary}\label{cor:ztexact}\index{funtor!de Zettl}
The functor $E\mapsto E^r$, $\pi\mapsto\pi^r$, from the category 
of \cts into the category of \css, is exact. More precisely: if 
\begin{equation*} \xymatrix
   {0\ar@{->}[r]&F_1 \ar@{->}[r]^{\phi}&F_2\ar@{->}[r]^{\psi}&F_3\ar@{->}[r]&0}
\end{equation*}  
is an exact sequence of \cts, then the sequence:  
\begin{equation*} \xymatrix
    {0\ar@{->}[r]&F_1^r\ar@{->}[r]^{\phi^r}&F_2^r
    \ar@{->}[r]^{\psi^r}&F_3^r\ar@{->}[r]&0}
\end{equation*} 
also is exact.
\end{corollary}
\begin{corollary}\label{cor:imagem}
If $\pi :E\to F$ is a \hm of \cts, then $\pi (E)$ 
is closed in $F$. The ideals of a \ct $E$ are exactly the kernels 
of the \hms defined on $E$.
\end{corollary}


\section{Applications}
\subsection{\texorpdfstring{$C^*$}{C*}-algebras associated with Fell bundles}
\par The proof of Theorem~1.1 of \cite{flau} relies on the existence
of a certain inner product (see Corollary~\ref{cor:uff!}
below), although no proof is included there of the 
fact that such inner product is indeed positive. In the following
lines we provide such a proof, and we refine the above mentioned
result.  

\par Recall that a right ideal $\me=(E_t)_{t\in G}$ of a Fell bundle
$\mb=(B_t)_{t\in G}$ is a sub-Banach bundle of $\mb$ such that
$\me\mb\subseteq\me$. 

\par Given a right Hilbert $B$-module $X$, let denote by $D_X$ the 
cone of finite sums $\sum_i\pr{x_i}{x_i}\subseteq B^+$. It is clear
that if $\{X_\lambda\}_{\lambda\in\Lambda}$ is a family of right
Hilbert $B$-modules and $X:=\oplus_\lambda X_\lambda$ (direct sum of
Hilbert modules), then
$\sum_\lambda D_{X_\lambda}\subseteq D_X$ -with equality if $\Lambda$
is finite- and $\sum_\lambda D_{X_\lambda}$ is dense in $D_X$.     
\par Similarly, for the right ideal $\mathcal{E}$ of the Fell bundle
$\mb$, we define $D_\me:=\{\sum_{i=1}^nc_i^*c_i: n\in \nt,
c_i\in \me, \forall i\}\subseteq B_e^+$. Then we have:
\begin{lemma}\label{lem:pre2}
Let $\mathcal{E}=(E_t)_{t\in G}$ be a right ideal of the Fell bundle
$\mb=(B_t)_{t\in G}$. Then $\gen(\mathcal{\me^*\me}\cap B_e)$ is dense
in $B_e$ if and only if the cone 
$D_\me$ satisfies the following property: 
\begin{equation}\label{eq:prop}
\forall b\in B_e,\ 
\epsilon>0,\text{ there exists }d\in D_\me\text{ such that
}\norm{d}\leq 1\text{ and } 
\norm{b-bd}<\epsilon.
\end{equation} 
\end{lemma}   
\begin{proof}
Suppose that $b\in B_e$ is such that for any $\epsilon>0$ there exists
$d\in D_\me$ such that $\norm{b-bd}<\epsilon$. Since $D_\me\subseteq
\gen(\me^*\me\cap B_e)$ and the latter is an ideal in $B_e$, we conclude
that $b\in\overline{\gen}(\me^*\me\cap B_e)$. Then $\gen(\me^*\me\cap
B_e)$ is dense in $B_e$ whenever $D_\me$ satisfies \eqref{eq:prop}.  
Note now that $D_\me=\sum_{t\in G}D_{E_t}$, which is dense in
$D_E$, where $E:=\oplus_{t\in G}E_t$. Thus $D_\me$ satisfies
\eqref{eq:prop} if and only if that property holds
for $D_E$. Assume that $\gen(\me^*\me\cap B_e)$ is dense in
$B_e$. Then $E$ is a full Hilbert module over $B_e$, and therefore it
satisfies \eqref{eq:prop} by \cite[(ii) of Lemma~7.2]{l}. 
 \end{proof}

\begin{lemma}\label{lem:2}
Let $\mb=(B_t)_{t\in G}$ be a Fell bundle over the locally compact
group $G$, $\ma=(A_t)$ a sub-Fell bundle of 
$\mb$, and $\me=(E_t)$ a right ideal of $\mb$ such that
$\ma\subseteq\me$, $\me\me^*\subseteq\ma$ and $\gen(\me^*\me\cap B_e)$
is dense in $B_e$. 
If $\xi\in L^1(\me)$, then
$\xi*\xi^*$ can be arbitrarily approximated in $L^1(\ma)$ by a finite
sum $\sum_{j=1}^m\eta_j*\eta_j^*$, where $\eta_j\in L^1(\ma)$,
$\forall j=1,\ldots,m$.    
\end{lemma}
\begin{proof}
  We will suppose that $\xi \in C_c(\me),$ which is clearly
  enough. Since $C_0(\me)$ is a nondegenerate right Banach
  $B_e$-module, given a positive integer $n$ there exists $b_n\in B_e$ 
  such that $\|\xi - \xi b_n\|<1/n$ and $0\leq b_n\leq 1.$ Then we can 
  find $c_n\in D_\me$ such that $\|b_n^{1/2}-b_n^{1/2}c_n\|<1/n.$ Set
  $d_n:= b_n^{1/2}c_n b_n^{1/2}$ and note that $d_n\in D_\me$ because
  $\me$ is a right ideal. The continuity of the operations imply $\|
  b_n - d_n \| \to 0$ and $\| \xi - \xi d_n \|_1\to 0.$ Thus
  $\|\xi*\xi^* - \xi d_n*\xi^*\|_1\to 0.$ 
  \par Now for every $n$ there exist $u_1,\ldots, u_{m_n}\in \me$ such that
  $d_n=\sum_{j=1}^{m_n} {u_j}^* u_j.$ Thus $\xi d_n * \xi^* =
  \sum_{j=1}^{m_n} (\xi {u_j}^* u_j)*\xi^*=\sum_{j=1}^{m_n} (\xi
  {u_j}^*)*(\xi {u_j}^*)^*$ and, as $\me$ is a right ideal, $\xi
  {u_j}^*\in C_c(\mathcal{A}).$ This completes the proof.  
 \end{proof}

\begin{corollary}\label{cor:uff!}
Under the assumptions of Lemma~\ref{lem:2}, let
$\norm{\ }_\ma:L^1(\ma)\to [0,\infty)$ be the maximal $C^*$-norm 
of $L^1(\ma)$. Then $L^1(\me)\times L^1(\me)\to L^1(\ma)$ given by
$(\xi,\eta)\mapsto \xi*\eta^*$ is an inner product. 
\end{corollary}
\begin{corollary}\label{cor:uff!2}
Under the assumptions of Lemma~\ref{lem:2}, the map $\varphi:
\msn(L^1(\mb))\to \msn(L^1(\mb))$ given by
$\beta\mapsto\beta|_{L^1(\ma)}$ is an isomorphism of partially ordered
sets that sends the maximal and reduced norms on $L^1(\mb)$ to the
maximal and reduced norms on $L^1(\ma)$ respectively, and such that
$\varphi(\mn(L^1(\mb)))=\mn(L^1(\ma))$. Moreover, the Hausdorff
completions of $L^1(\mb)$ and $L^1(\ma)$ with respect to $\beta$ and
$\varphi(\beta)$ respectively are Morita-Rieffel equivalent.
\end{corollary}
\begin{proof}
We only have to prove the correspondence between the reduced
$C^*$-norms, but this is the content of \cite{flau}.  
\end{proof}
 
\subsection{Tensor products of \texorpdfstring{$C^*$}{C*}-trings}\label{sec:tenstrings}
In the present section we apply the previous results to the study of
tensor products of \cts.
Maximal and minimal tensor product for TROs were constructed in \cite{karu} using linking algebras, but we define tensor products of \cts $E$ and $F$ using the tensor products of $E^r$ and $F^r$.
The main result is Theorem~\ref{thm:main1}.
\par From now on the algebraic tensor product of the 
$\C$-vector spaces $E_1,\ldots ,E_n$ will be denoted by
$E_1\bigodot\ldots\bigodot E_n$, or just by
$\bigodot_{j=1}^nE_j$. 
Let $E_{ij}$, $F_i$ be complex vector spaces, $\forall i=1,\ldots ,m$,
  $j=1,\ldots ,n$, and suppose that  
  $\al_i:\prod_{j=1}^n E_{ij}\to F_i$ is a $n$-linear map, for each 
 $i=1,\ldots ,m.$  Then it is clear that there exists a unique
  $n$-linear map $\al:= 
  \al_1\odot\cdots\odot\al_m: 
  \prod_{j=1}^n\bigodot_{i=1}^mE_{ij}\to\bigodot_{i=1}^mF_i$ 
  such that $\al (\odot_{i=1}^me_{i1},\ldots
  ,\odot_{i=1}^me_{in})=\odot_{i=1}^m\al_i(e_{i1},\ldots
  ,e_{in})$. Using this fact we have the following result, whose
  straightforward proof is left to the reader.   
\begin{proposition}\label{prop:tprod}
If $(E,\mu )$, $(F,\nu )$ are $*$-trings, then $(\bats{E}{F}
,\mu\odot\nu )$ is also a $*$-tring.
Furthermore, if $(E,A,\pr{\,}{\,}_A)$ and $(F,B,\pr{\,}{\,}_B),$ are full basic triples associated to $(E,\mu )$ and $(F,\nu ),$ respectively, then $(\bats{E}{F},\bats{A}{B},\pr{\,}{\,}_A\odot \pr{\,}{\,}_B)$ is a full basic triple associated to $(\bats{E}{F},\mu\odot\nu ).$ 
\end{proposition}
\begin{definition}\label{defn:norms}
A $C^*$-tensor product of two $*$-trings $(E,\mu ,\norm{\cdot})$ and 
$(F,\nu ,\norm{\cdot})$ is a completion of the corresponding algebraic
tensor product $(\bats{E}{F},\mu\odot\nu )$ with respect to a 
$C^*$-norm. If $\gamma$ is such a $C^*$-norm, we denote by 
\bts{E}{\gamma}{F} the corresponding $C^*$-tensor product.
\end{definition}
\begin{definition}\label{defn:nuclearct}
We say that a \ct $E$ is nuclear if for every \ct $F$ there exists
just one $C^*$-tensor product \bts{E}{}{F}.
\end{definition}
\par We will see next that $\msn(\bats{E}{F})=\msn(\bats{E_p}{F})$,
which implies, in particular, that a $C^*$-tring $E$ is 
nuclear if and only if $E_p$ is nuclear. 
\begin{proposition}\label{prop:ds}
Let $E$ be a $*$-tring, and $F_1,F_2$ ideals of $E$ such that
$E=F_1\oplus 
F_2$. If $\ga\in\msn(E)$, and $x=y+z$, with $y\in F_1$ and $z\in F_2$,
then $\ga(x)=\max\{\ga(y),\ga(z)\}$.  
\end{proposition}
\begin{proof}
Since $\ga(x)=\sup\{\ga((x,u,u)):\,u\in E,\ga(u)\leq 1\}$, it follows
that $\ga(x)\geq\ga(y)$ and $\ga(x)\geq\ga(z)$, so
$\ga(x)\geq\max\{\ga(y),\ga(z)\}$. To prove the converse
inequality, let us first introduce the following notation. For $u\in
E$ let $u_0:=z$, $u_n:=(u_{n-1},u_{n-1},u_{n-1})$ if $n\geq 1$. Then
we have that $\ga(u_n)=\ga(u_{n-1})^3$, $\forall n\geq 1$, so
$\ga(u_n)=\ga(u)^{3^n}$, $\forall n\geq 0$. Since $(E,F_1,F_2)=0$, it
follows that $x_n=y_n+z_n$. Thus:
$\ga(x)=\ga(x_n)^{1/3^n}=\ga(y_n+z_n)^{1/3^n}\leq
(\ga(y_n)+\ga(z_n))^{1/3^n}=(\ga(y)^{3^n}+\ga(z)^{3^n})^{1/3^n}
\stackrel{n}{\to}\max\{\ga(y),\ga(z)\}$, whence $\ga(x)\leq
\max\{\ga(y),\ga(z)\}$.     
\end{proof}

\begin{corollary}\label{cor:E}
Let $E$ and $F$ be \cts. Then
$\msn(\bats{E}{F})=\msn(\bats{E_p}{F})$ and
$\mn(\bats{E}{F})=\mn(\bats{E_p}{F})$. Consequently a \ct $E$ is
nuclear if and only if $E_p$ is nuclear. 
\end{corollary}
\par Our aim is to prove that there is an isomorphism between
$\mn(\bats{E}{F})$ and $\mn(\bats{E^r}{F^r})$. The key step is to show
that each $C^*$-norm on $\bats{E_0^r}{F_0^r}$ has unique extension to a
$C^*$-norm on $\bats{E^r}{F^r}$. 

\begin{lemma}\label{lem:tensor product norms}
 Let $I$ and $J$ be *-ideals (not necessarily closed) of the C*-algebras $A$ and $B,$ respectively.
 Then the map $\Theta\colon \mathcal{N}(A\odot B)\to \mathcal{N}(I\odot J),$ $\gamma\mapsto \gamma|_{I\odot J},$ is an order preserving surjection.
 If, in addition, $I$ and $J$ are dense in $A$ and $B,$ respectively,
 then $\Te$ is a bijection. 
\end{lemma}
\begin{proof}
  Clearly $\Theta$ is order preserving.
  Fix $\delta\in \mathcal{N}(I\odot J).$ 
 Given $a\in A$ and $z=\sum_{j=1}^n x_i\odot y_j\in I\odot J$, define
 $w:=\sum_{j=1}^n (\|a\|^2 - a^*a)^{1/2}x_i\odot
 y_j\in A\odot B.$ In case $A$ is unital it is clear that $w\in I\odot
 J.$  
  If $A$ is not unital, $I$ is an ideal of the unitization of $A,$ so
  $w\in I\odot J$ in any case. 
  Then 
  \begin{equation*}
    \|a\|^2 z^*z - (\sum_{j=1}^n ax_i\odot y_j)^*(\sum_{j=1}^n ax_i\odot y_j) 
      = w^*w\in (I\otimes_\delta J)^+
  \end{equation*}
  and $\delta(\sum_{j=1}^n ax_i\odot y_j)\leq \|a\|
  \delta(\sum_{j=1}^n x_i\odot y_j).$ Similarly, if $b\in B$, we also
  have $\delta(\sum_{j=1}^n x_i\odot by_j)\leq \|b\|
  \delta(\sum_{j=1}^n x_i\odot y_j).$ Thus $\delta((a\odot
  b)z)\leq\norm{a}\,\norm{b}\delta(z)$, $\forall a\in A$, $b\in B$ and
  $z\in I\odot J$. Therefore, according to
  \ref{prop:wbijnorms}, the map $\delta':A\odot B\to \R$ such that
  $\delta'(c):=\sup\{\delta(cz):\delta(z)\leq 1\}$ is a $C^*$-seminorm
  on $A\odot B$ that extends $\delta$.   
  In case $I$ and $J$ are dense in $A$ and $B,$ respectively, $I\odot J$ is dense in $A\odot B$ with respect to any C*-norm \cite[Corollary T.6.2]{w}.
  Thus $\Te$ is injective.
\end{proof}

\begin{proposition}\label{prop:tensor pos C*trings}
 Let $E$ and $F$ be positive C*-trings and consider the admissible full basic triples $(E,E^r_0,\pr{\,}{\,}_r^E)$ and $(F,F^r_0,\pr{\,}{\,}_r^F)$ given by Theorem~\ref{thm:*level}.
 Then the full basic triple $(\bats{E}{F},\bats{E^r_0}{F^r_0},\pr{\,}{\,}_r^E\odot \pr{\,}{\,}_r^F)$ is admissible.
 Furthermore, $\bats{E}{F}$ is positive and
 \begin{equation*}
    \msn_{cs}^{\pr{\,}{\,}_r^E\odot \pr{\,}{\,}_r^F}(\bats{E^r_0}{F^r_0})
      = \msn(\bats{E^r_0}{F^r_0})
 \end{equation*}
\end{proposition}
\begin{proof}
  To simplify our notation we denote $[\, ,\,]$ the map
  $\pr{\,}{\,}_r^E\odot \pr{\,}{\,}_r^F.$ Note 
  $\bats{E^r_0}{F^r_0}$-module is admissible because it is a *-subalgebra of the C*-closable *-algebra $\bats{E^r}{F^r}.$ We will show that
  $\bats{E}{F}$ is a positive $\bats{E^r_0}{F^r_0}$-module. 
  Lemma~\ref{lem:tensor product norms} implies there is a maximal
  C*-norm on $\bats{E^r_0}{F^r_0}$, namely the restriction of the maximal
  C*-norm of $\bats{E^r}{F^r}.$ 
  The comments preceding Lemma~\ref{lem:pos} imply that, to show $\bats{E}{F}$ is positive, it suffices to prove that $[u,u]\geq 0$ in the maximal tensor product $\tmax{E^r}{F^r}.$
  Given $u=\sum_{j=1}^n x_j\otimes y_j\in \bats{E}{F}$ we have 
  \begin{equation*}
    [u,u]
      =\sum_{j,k=1}^n \pr{x_j}{x_k}^E_r\odot \pr{y_j}{y_k}^r_F.
  \end{equation*}
  Then Lemmas 4.2 and 4.3 of \cite{l} give the desired result.
  
  To show $[u,u]=0$ implies $u=0$ we use the linking algebras $\link{E}$ and $\link{F}$ and the linear maps
  \begin{align*}
    \alpha \colon \bats{E}{F}\to \bats{\link{E}}{\link{F}},
       \ x\odot y\mapsto \bigl( \begin{smallmatrix}0&x\\0&0\end{smallmatrix}\bigr) \odot \bigl( \begin{smallmatrix}0&y\\0&0\end{smallmatrix}\bigr),\\
    \beta \colon \bats{\link{E}}{\link{F}} \to \bats{E}{F},\
       \bigl( \begin{smallmatrix}x_{11}&x_{12}\\x_{21}&x_{22}\end{smallmatrix}\bigr) \odot \bigl( \begin{smallmatrix}y_{11}&y_{12}\\y_{21}&y_{22}\end{smallmatrix}\bigr) \mapsto x_{12}\odot y_{12}, \\
    \gamma  \colon \bats{E}{F} \to \bats{\link{E}}{\link{F}},
       \ a\odot b \mapsto \bigl( \begin{smallmatrix}0&0\\0&a\end{smallmatrix}\bigr) \odot \bigl( \begin{smallmatrix}0&0\\0&b\end{smallmatrix}\bigr).
  \end{align*}
  Then $\alpha(u)^*\alpha(u)=\gamma([u,u])=0,$ so $\alpha(u)=0$ and $u=\beta(\alpha(u))=0.$
\end{proof}

\begin{theorem}\label{thm:main1}
Let $E$ and $F$ be $C^*$-ternary rings. Then every set among the
partially ordered sets $\mn(\bats{E^l}{F^l})$, $\mn(\bats{E}{F})$ and
$\mn(\bats{E^r}{F^r})$ is isomorphic to each other. Besides, if
$\ga\in\mn(\bats{E}{F})$ and $\gamma^l$ and
$\gamma^r$ are the corresponding $C^*$-norms on
$\mn(\bats{E^l}{F^l})$ and $\mn(\bats{E^r}{F^r})$ respectively, then
$\bts{E}{\ga}{F}$ is a Morita-Rieffel equivalence bimodule between
$\bts{E^l}{\ga^l}{F^l}$ and $\bts{E^r}{\ga^r}{F^r}$.  
\end{theorem}
\begin{proof}

Proposition~\ref{prop:tensor pos C*trings} together with
Corollary~\ref{cor:possc} imply $\mn(\bats{E}{F})$ is isomorphic (as a partially ordered set) to 
$\mn(\bats{E}{F})_0^r).$
By \ref{prop:tprod} the posets  
$\mn(\bats{E}{F})_0^r)$ and $\mn(\bats{E_0^r}{F_0^r})$ are
 isomorphic, and the latter is isomorphic to
 $\mn(\bats{E^r}{F^r})$ by Lemma~\ref{lem:tensor product norms}. Thus
 $\mn(\bats{E}{F})\cong \mn(\bats{E^r}{F^r})$. Similarly we have
 $\mn(\bats{E}{F})\cong \mn(\bats{E^l}{F^l})$.  
\end{proof}

\begin{corollary}\label{cor:maxmin}
Let $E$ and $F$ be \cts. Then there exist a maximum $C^*$-norm 
$\norm{\cdot}_{\max}$ on \bats{E}{F}, and a minimum $C^*$-norm 
$\norm{\cdot}_{\min}$ on \bats{E}{F}, and  
\begin{equation*}
  \big(\bts{E}{\max}{F}\big)^l
    =\bts{E^l}{\max}{F^l},
    \hspace*{1cm} 
  \big(\bts{E}{\max}{F}\big)^r
    =\bts{E^r}{\max}{F^r},
\end{equation*}
\begin{equation*}
  \big(\bts{E}{\min}{F}\big)^l
    =\bts{E^l}{\min}{F^l}
  \hspace*{1cm} 
  \big(\bts{E}{\min}{F}\big)^r
    =\bts{E^r}{\min}{F^r}.
\end{equation*}
\end{corollary}

\begin{corollary}[\textit{cf.} \mbox{\cite[Theorem 6.5]{karu}}]\label{cor:nuclear}
The following assertions are equivalent for a \ct $E$:
\be
 \item $E$ is a nuclear \ct (\ref{defn:nuclearct}).
 \item $E^l$ is a nuclear \cs.
 \item $E^r$ is a nuclear \cs. 
\ee 

\end{corollary}

\par The equivalence between 2. and 3. in \ref{cor:nuclear} is exactly 
the following well-known result (\cite{beer}, \cite{z1}):
\textit{if $A$ and $B$ are two Morita-Rieffel equivalent \css then $A$ is
nuclear if and only if so is $B$.}

\subsection{Exact \texorpdfstring{$C^*$}{C*}-trings}\label{subsec:exactrings}

To end the section we introduce the notion of exact \ct, extending the notion of exact TRO of \cite{karu}, and we 
prove a result similar to Corollary \ref{cor:nuclear}. The reader is
referred to \cite{wass} for the theory of exact \css.   

Suppose that 
$ \xymatrix
{0\ar@{->}[r]&F_1\ar@{->}[r]^{\phi}&F_2\ar@{->}[r]^{\psi}&F_3\ar@{->}[r]&0}$
is an \es of $C^*$-trings, that is, $\phi$ and   
$\psi$ are \hms of \cts, $\phi$ is injective, $\psi$ is
surjective, and $\ker\psi =\phi (F_1)$. Let $E$ be a \ct. Then the 
sequence 
\begin{equation*}
  \xymatrix
   {0\ar@{->}[r]&\bats{E}{F_1}\ar@{->}[r]^{id\odot\phi}
                &\bats{E}{F_2}\ar@{->}[r]^{id\odot\psi}
                &\bats{E}{F_3}\ar@{->}[r]&0}
\end{equation*}
also is exact. We have an inclusion
\begin{equation*}
 (\bats{E}{F_2})/(\bats{E}{F_1})\inc (\tmin{E}{F_2})/(\tmin{E}{F_1})
\end{equation*}
and the latter quotient is a \ct. Then there exists a $C^*$-norm $\ga$ on
\bats{E}{F_3} such that  
\begin{equation*} \xymatrix
  {0\ar@{->}[r]&\tmin{E}{F_1}\ar@{->}[r]^{id\otimes\phi}
              &\tmin{E}{F_2}\ar@{->}[r]^{id\otimes\psi}
              &\bts{E}{\ga}{F_3}\ar@{->}[r]&0}
\end{equation*}
is exact. Since $\ga$ is greater or equal to the minimum norm, the identity map on  
\bats{E}{F_3} extends to a surjective \hm
$\bts{E}{\ga}{F_3}\to\tmin{E}{F_3}$. 

\begin{definition}\label{defn:exacttrings}
We say that a \ct $E$ is  exact if for each \es 
\begin{equation*}
  \xymatrix
   {0\ar@{->}[r]&F_1\ar@{->}[r]&F_2\ar@{->}[r]&F_3\ar@{->}[r]&0}
\end{equation*}
of \cts we have that 
\begin{equation*}
  \xymatrix
   {0\ar@{->}[r]&\tmin{E}{F_1}\ar@{->}[r]&\tmin{E}{F_2}\ar@{->}[r]&
                 \tmin{E}{F_3}\ar@{->}[r]&0}
\end{equation*}
also is exact.
\end{definition}

\begin{proposition}\label{prop:exex}
Let $E$ and $F$ be \cts, and suppose that $G$ is an ideal of  
$F$ (Definition \ref{defn:ideals}). Then 
\begin{equation*} \xymatrix
   {0\ar@{->}[r]&\tmin{E}{G}\ar@{->}[r]&\tmin{E}{F}\ar@{->}[r]&
    \tmin{E}{(F/G)}\ar@{->}[r]&0}
\end{equation*} 
is exact if and only if the following sequence is exact: 
\begin{equation*} \xymatrix
{0\ar@{->}[r]&\tmin{E^r}{G^r}\ar@{->}[r]&
              \tmin{E^r}{F^r}\ar@{->}[r]&
              \tmin{E^r}{\big(F^r/G^r\big)}\ar@{->}[r]&0} 
\end{equation*}
\end{proposition}
\begin{proof}
Suppose first that the sequence below is exact:  
\begin{equation*} \xymatrix@1
   {0\ar[r]&\tmin{E}{G}\ar[r]&\tmin{E}{F}\ar[r]&
                 \tmin{E}{(F/G)}\ar[r]&0}
\end{equation*} 
By Corollaries \ref{cor:maxmin} and \ref{cor:ztexact}, we
have the following commutative diagram  
\begin{equation*} \xymatrix{
         0\ar[r]&\big(\tmin{E}{G}\big)^r\ar[r]\ar[d]_-{\cong}
                &\big(\tmin{E}{F}\big)^r\ar[r]\ar[d]_-{\cong}
                &\big(\tmin{E}{(F/G)}\big)^r\ar[r]\ar[d]_-{\cong}
                &0\\
         0\ar[r]&\tmin{E^r}{G^r}\ar[r]\ar[d]_=
                &\tmin{E^r}{F^r}\ar[r]\ar[d]_=
                &\tmin{E^r}{(F/G)^r}\ar[r]\ar[d]_-{\cong}
                &0 \\
         0\ar[r]&\tmin{E^r}{G^r}\ar[r]
                &\tmin{E^r}{F^r}\ar[r]
                &\tmin{E^r}{F^r/G^r}\ar[r]
                &0        
     }
\end{equation*}
Since the upper two rows are exact, the third one also is exact.
\par To prove the converse, note first that  
\begin{equation*} \xymatrix@1
   {0\ar[r]&\tmin{E}{G}\ar[r]&\tmin{E}{F}\ar[r]&
                 \big(\tmin{E}{F}\big)/\big(\tmin{E}{G}\big)\ar[r]&0}
\end{equation*} 
is  exact, and $\big(\tmin{E}{F}\big)/\big(\tmin{E}{G}\big)$ is a 
$C^*$-completion of the ternary ring $\bats{E}{(F/G)}$. Denoting the
corresponding $C^*$-norm by $\ga$, we have a surjective \hm 
$\phi :\bts{E}{\ga}{(F/G)}\to \tmin{E}{(F/G)}$ which extends the
identity on $\bats{E}{(F/G)}$. Now, applying the exact functor $E\mapsto
E^r$ we obtain the commutative diagram with exact rows that follows: 
\begin{equation*} \xymatrix{
         0\ar[r]&\tmin{E^r}{G^r}\ar[r]\ar[d]_=
                &\tmin{E^r}{F^r}\ar[r]\ar[d]_=
                &\bts{E^r}{\ga}{F^r/G^r}\ar[r]\ar[d]_{\phi^r}
                &0  \\
          0\ar[r]&\tmin{E^r}{G^r}\ar[r]
                &\tmin{E^r}{F^r}\ar[r]
                &\tmin{E^r}{F^r/G^r}\ar[r]
                &0        
     }
\end{equation*}
It follows that the \hm $\phi^r$ is an isomorphism. 
\end{proof}

\begin{corollary}[\textit{cf.} \mbox{\cite[Theorem 6.1]{karu}}]\label{cor:exex}
A \ct $E$ is exact (\ref{defn:exacttrings}) if and only if $E^r$ is an
exact \cs. 
\end{corollary}
\begin{proof}Immediate from Proposition \ref{prop:exex}
\end{proof}

\par As previously for nuclear \css, we easily obtain from
\ref{cor:exex} the following known result(\cite{ngess}): \textit{if
$A$ and $B$ are Morita-Rieffel equivalent \css, then $A$ is exact if and only 
if $B$ is exact.}

\bibliographystyle{amsplain}

\end{document}